\documentclass{birkjour}
\usepackage{graphicx}
\usepackage{amscd}
\usepackage{amsmath}
\usepackage{amsxtra}
\usepackage{amsfonts}
\usepackage{amssymb}
\usepackage{xcolor}
\usepackage{mathrsfs}  
\usepackage{leftindex}
\usepackage[euler]{textgreek}

\newtheorem{theorem}{Theorem}[section]
\newtheorem{corollary}[theorem]{Corollary}
\newtheorem{lemma}[theorem]{Lemma}
\newtheorem{proposition}[theorem]{Proposition}

\newtheorem{conjecture}[theorem]{Conjecture}
\theoremstyle{definition}
\newtheorem{definition}[theorem]{Definition}

\newtheorem{remark}[theorem]{Remark}

\theoremstyle{remark}

\renewcommand{\theclaim}{\textup{\theclaim}}

\renewcommand{\theenumi}{\roman{enumi}}
\renewcommand{\labelenumi}{(\theenumi)}
\numberwithin{equation}{section}

\def\openone

{\mathchoice

{\hbox{\upshape \small1\kern-3.3pt\normalsize1}}

{\hbox{\upshape \small1\kern-3.3pt\normalsize1}}

{\hbox{\upshape \tiny1\kern-2.3pt\SMALL1}}

{\hbox{\upshape \Tiny1\kern-2pt\tiny1}}}

\makeatletter

\newbox\ipbox

\newcommand{\ip}[2]{\left\langle #1\, , \,#2\right\rangle}
\newcommand{\diracb}[1]{\left\langle #1\mathrel{\mathchoice

{\setbox\ipbox=\hbox{$\displaystyle \left\langle\mathstrut
#1\right.$}

\vrule height\ht\ipbox width0.25pt depth\dp\ipbox}

{\setbox\ipbox=\hbox{$\textstyle \left\langle\mathstrut
#1\right.$}

\vrule height\ht\ipbox width0.25pt depth\dp\ipbox}

{\setbox\ipbox=\hbox{$\scriptstyle \left\langle\mathstrut
#1\right.$}

\vrule height\ht\ipbox width0.25pt depth\dp\ipbox}

{\setbox\ipbox=\hbox{$\scriptscriptstyle \left\langle\mathstrut
#1\right.$}

\vrule height\ht\ipbox width0.25pt depth\dp\ipbox}

}\right. }

\newcommand{\dirack}[1]{\left. \mathrel{\mathchoice

{\setbox\ipbox=\hbox{$\displaystyle \left.\mathstrut
#1\right\rangle$}

\vrule height\ht\ipbox width0.25pt depth\dp\ipbox}

{\setbox\ipbox=\hbox{$\textstyle \left.\mathstrut
#1\right\rangle$}

\vrule height\ht\ipbox width0.25pt depth\dp\ipbox}

{\setbox\ipbox=\hbox{$\scriptstyle \left.\mathstrut
#1\right\rangle$}

\vrule height\ht\ipbox width0.25pt depth\dp\ipbox}

{\setbox\ipbox=\hbox{$\scriptscriptstyle \left.\mathstrut
#1\right\rangle$}

\vrule height\ht\ipbox width0.25pt depth\dp\ipbox}

} #1\right\rangle}

\newcommand{\beq}{\begin{equation}}

\newcommand{\eeq}{\end{equation}}

\newcommand{\cj}[1]{\overline{#1}}

\newcommand{\bz}{\mathbb{Z}}

\newcommand{\brr}{\mathbb{R}}
\newcommand{\bc}{\mathbb{C}}

\newcommand{\bn}{\mathbb{N}}

\def\blfootnote{\xdef\@thefnmark{}\@footnotetext}

\input xy
\xyoption{all}
\usepackage{amssymb}

\newcommand{\supp}{\operatorname*{supp}}
\newcommand{\norm}[1]{\lvert \lvert#1\rvert \lvert }

\def\-{^{-1}}

\def\ty{\emptyset}

\begin{document}

\title[Self-adjoint extensions of the partial differential operators]{Commuting self-adjoint extensions of the partial differential operators on disconnected sets}

\author{Piyali Chakraborty}
\address{University of Central Florida\br
	Department of Mathematics\br
	4000 Central Florida Blvd.\br
	P.O. Box 161364\br
	Orlando, FL 32816-1364\br
	U.S.A.} \email{Piyali.Chakraborty@ucf.edu}

\author{Dorin Ervin Dutkay${}^*$}
\address{University of Central Florida\br
	Department of Mathematics\br
	4000 Central Florida Blvd.\br
	P.O. Box 161364\br
	Orlando, FL 32816-1364\br
U.S.A.} \email{Dorin.Dutkay@ucf.edu}

\subjclass{Primary 47E05; Secondary 42A16}
\keywords{differential operator, self-adjoint operator, unitary group, Fourier bases, Fuglede conjecture, lattice}

\begin{abstract}
We investigate the existence of commuting self-adjoint extensions of the partial differential operators
$$
D_j=\frac{1}{2\pi i}\frac{\partial}{\partial x_j},
\qquad j=1,\dots,d,
$$
defined initially on $C_0^\infty(\Omega)$, where $\Omega\subset \mathbb{R}^d$ is an arbitrary open set, possibly disconnected and unbounded. This problem originates in a question of Segal and is closely related to Fuglede’s characterization of spectral sets through commuting self-adjoint extensions and orthogonal Fourier bases. While previous results of Fuglede and Pedersen established a complete correspondence for connected domains, the disconnected case presents new phenomena arising from interactions among distinct connected components.

We extend the spectral-theoretic framework to arbitrary open sets and obtain a necessary and sufficient condition for the existence of commuting self-adjoint extensions of the operators $\{D_j\}$. The characterization is expressed in terms of a Radon measure, a multiplicity function, and a measurable matrix field that encode the spectral decomposition of the resulting commuting family. In contrast with the connected case, multiplicities greater than one may occur, and generalized eigenfunctions involve component-dependent coefficients attached to the exponential functions $e_\lambda(x)=e^{2\pi i\lambda\cdot x}$.

We further establish a characterization of spectral sets in terms of strongly continuous unitary representations acting as groups of local translations. For arbitrary open sets, we show that the existence of a group of local translations is equivalent to the spectrality of the domain. This extends the classical connection between spectral pairs, Fourier analysis, and translation symmetries beyond the connected setting.

For open sets of finite measure with finitely many connected components, we prove that the associated spectral measure is atomic and that the joint spectrum is separated. We obtain an orthonormal basis consisting of piecewise exponential functions whose amplitudes may vary from one connected component to another. In the one-dimensional setting, where $\Omega$ is a countable union of intervals, we describe the corresponding unitary groups through a boundary transition matrix governing the propagation of translations across interval endpoints. Finally, for sets that tile $\mathbb{R}^d$ by a discrete subgroup, we characterize the associated pair measures and derive an explicit formula for the corresponding group of local translations, thereby extending a classical result of Fuglede from lattices to arbitrary discrete subgroups.

\end{abstract}
\maketitle
\tableofcontents
\newcommand{\Ds}{\mathsf{D}}
\newcommand{\Dmax}{\mathscr D_{\operatorname*{max}}}
\newcommand{\Dmin}{\Ds_{\operatorname*{min}}}
\newcommand{\dom}{\operatorname*{dom}}

\section{Introduction}
In 1958 Irving Segal posed the following problem to Bent Fuglede: Let $\Omega$ be an open set in $\brr^d$. Consider the partial differential operators $D_1=\frac1{2\pi i}\frac\partial{\partial x_1}$,$\dots$, $D_d=\frac1{2\pi i}\frac\partial{\partial x_d}$ defined on the space of $C^\infty$-differentiable, compactly supported functions on $\Omega$, denoted by $C_0^\infty(\Omega)$. Under what conditions these partial differential operators admit  commuting (unbounded) self-adjoint extensions $H_1,\dots,H_d$ on $L^2(\Omega)$, the commutation being understood in the sense of commuting spectral measures?

In his 1974 seminal paper \cite{Fug74}, Fuglede offered the following answer to Segal's question, for {\it connected domains} $\Omega$, of {\it finite measure} and which are Nikodym domains, meaning that the Poincar\'e inequality is satisfied:

\begin{theorem}\label{th2.1}\cite[Theorem I]{Fug74}	Let $\Omega\subset\brr^d$ be a finite measure open and connected Nikodym region. Denote by 
	$$e_\lambda(x)=e^{2\pi i\lambda\cdot x},\quad (x\in\brr^d,\lambda\in\brr^d).$$
	
	\begin{enumerate}
		\item[(a)] Let $H=(H_1,\dots,H_d)$ denote a commuting family (if any) of self-adjoint extensions $H_j$ of the differential operators $D_j$ on $L^2(\Omega)$, $j=1,\dots,n$. Then $H$ has a discrete spectrum, each point $\lambda\in \sigma(H)$ being a simple eigenvalue for $H$ with the eigenspace $\bc e_\lambda$, and hence $(e_\lambda)_{\lambda\in \sigma(H)}$ is an orthogonal basis for $L^2(\Omega)$. Moreover, the spectrum $\sigma(H)$ and the point spectrum $\sigma_p(H)$ are equal and  $$\sigma(H)=\sigma_p(H)=\{\lambda\in\brr^d : e_\lambda\in\mathscr D( H)\}.$$
		$\mathscr D(H)$ denotes the intersection of the domains of the operators $H_j$: $\mathscr D(H):=\cap_j \mathscr D(H_j)$.
		\item[(b)] Conversely, let $\Lambda$ denote a subset (if any) of $\brr^d$ such that $(e_\lambda)_{\lambda\in\Lambda}$ is an orthogonal basis for $L^2(\Omega)$. Then there exists a unique commuting family $H=(H_1,\dots,H_d)$ of self-adjoint extensions $H_j$ of $D_j$ on $L^2(\Omega)$ with the property that $\{e_\lambda : \lambda\in\Lambda\}\subset\mathscr D(H)$, or equivalently that $\Lambda=\sigma(H)$.
	\end{enumerate}
\end{theorem}

\begin{definition}
	\label{defsp}
	A subset $\Omega$ of $\brr^d$ of finite Lebesgue measure is called {\it spectral} if there exists a set $\Lambda$ in $\brr^d$, such that the family of exponential functions 
	$$\{e_\lambda : \lambda\in\Lambda\}$$
	forms an orthogonal basis for $L^2(\Omega)$. In this case, we say that $\Lambda$ is a {\it spectrum} for $\Omega$. 
	
	We say that a set $\Omega$ {\it tiles} $\brr^d$ by translations if  there exists a set $\mathcal T$ in $\brr^d$ such that $\{\Omega +t : t\in\mathcal T\}$ is a partition of $\brr^d$ up to measure zero. We also call $\mathcal T$ a {\it tiling set} for $\Omega$ and we say that $\Omega$ {\it tiles by} $\mathcal T$. 
\end{definition}

In these terms, when $\Omega$ is connected, of finite measure and a Nikodym domain, there exist commuting self-adjoint extensions of the partial differential operators $\{D_j\}$ if and only if $\Omega$ is spectral. Since this condition is a bit vague, Fuglede proposed his famous conjecture:
\begin{conjecture}{\bf[Fuglede's Conjecture]}
	\label{conFu}
	A subset $\Omega$ of $\brr^d$  of finite Lebesgue measure is spectral if and only if it tiles $\brr^d$ by translations.
\end{conjecture}
The conjecture was proved to be false in dimensions 3 and higher \cite{Tao04, FMM06}, but it is true for convex domains \cite{LM22}.

In \cite{Ped87}, Steen Pedersen improved Fuglede's result: he removed not only the Nikodym restriction but also the finite measure condition! Of course, the definition of a spectral set had to be modified for sets of infinite measure (because, in this case, the functions $e_\lambda$ are not $L^2$-integrable).

\begin{definition}\label{defp1}
	For a function $f\in L^1(\brr^d)$, let $\hat f$ be the classical Fourier transform 
	$$\hat f(\lambda)=\int_{\brr^d}f(t) e^{-2\pi i\lambda\cdot t}\,dt,\quad(\lambda\in\brr^d).$$

	Let $\Omega$ be a measurable subset of $\brr^d$ and let $\mu$ be a positive Radon measure on $\brr^d$. We will say that $(\Omega,\mu)$ is a {\it spectral pair} if (1) for each $f\in L^1(\Omega)\cap L^2(\Omega)$, the continuous function $\lambda\rightarrow \hat f(\lambda)=({f},{e_\lambda})$ satisfies $\int |\hat f|^2\,d\mu<\infty$, and (2) the map $f\rightarrow \hat f$ of $L^1(\Omega)\cap L^2(\Omega)\subset L^2(\Omega)$ into $L^2(\mu)$ is isometric and has dense range. 
	
	This map then extends by continuity to an isometric isomorphism
	$$\mathscr F:L^2(\Omega)\rightarrow L^2(\mu).$$

	The set $\Omega$ is called {\it spectral} if there is a measure $\mu$ such that $(\Omega,\mu)$ is a spectral pair; we also say that $\mu$ is a {\it pair measure} for $\Omega$.  When $\Omega$ has finite measure, the two definitions of a spectral set coincide \cite[Corollary 1.11]{Ped87}.
\end{definition}

With this definition of a spectral set, Pedersen improved Fuglede's theorem, by removing the Nikodym restriction, and the finite measure restriction, but keeping the condition that $\Omega$ be connected. 
\begin{theorem}\label{thp5}\cite[Theorem 2.2]{Ped87}
	Let $\Omega$ be an open and connected subset of $\brr^d$. The partial differential operators $\{D_j\}$ have commuting self-adjoint extensions if and only if $\Omega$ is a spectral set. 
\end{theorem}

Also, the existence of commuting self-adjoint extensions of the operators $\{D_j\}$ can be formulated in terms of unitary groups on $L^2(\Omega)$ that act locally as translations. Indeed, if $\{H_j\}$ are some commuting self-adjoint operators on a Hilbert space $\mathfrak H$, then using functional calculus
\begin{equation}
	\label{equ}
	U(t_1,\dots,t_d)=\exp(2\pi i(t_1 H_1+t_2H_2+\dots t_dH_d)),\quad( (t_1,\dots,t_d)\in\brr^d),
\end{equation}
defines a strongly continuous unitary  group representation of $\brr^d$ on the Hilbert space $\mathfrak H$. The converse is also true, by the Generalized Stone Theorem: each strongly continuous unitary representation of $\brr^d$ is of this form, for some commuting self-adjoint operators $\{H_j\}$. In the particular case when the operators $\{H_j\}$ are commuting self-adjoint extensions of the partial differential operators $\{D_j\}$, the group representation $\{U(t)\}$ has a certain property, called the {\it integrability property}, that it behaves  as translation inside each connected component, at small scales:

\begin{definition}
	\label{defp5}
	Let $\Omega$ be an open subset of $\brr^d$. We say that $\Omega$ has {\it the integrability property} if there exists a strongly continuous unitary representation $U$ of $\brr^d$ on $L^2(\Omega)$ which satisfies the following condition: for every $x$ in $\Omega$, for all neighborhoods $V$ of $x$ in $\Omega$ and all $\epsilon>0$ such that $V+t\subset\Omega$ for all $t\in\brr^d$ with $\|t\|<\epsilon$, we have 
	$$(U(t)f)(y)=f(y+t),\mbox{ for all $y\in V$, $t\in\brr^d$ with $\|t\|<\epsilon$ and }f\in L^2(\Omega).$$
	We also say that $\{U(t)\}$ has the {\it integrability property}.
	
For a function $f$ on $\brr^d$ and $t\in \brr^d$, we use the notation:
	\begin{equation}
		\label{eqT}
		T(t)f(x)=f(x+t),\quad (x+t\in \mbox{domain}(f)).
	\end{equation}
	
\end{definition}

Palle Jorgensen and Steen Pedersen proved the following equivalence:

\begin{theorem}\label{thai} \cite[Lemma 1]{Jor82},\cite[Proposition 1.2]{Ped87}, \cite[Theorem 5.7]{CDJ25} Let $\Omega$ be an open set in $\brr^d$. The following statements are equivalent:
	\begin{enumerate}
		\item There exists $H=(H_1,\dots,H_d)$ a commuting family of self-adjoint restrictions $H_j$ of $D_j$ on $L^2(\Omega)$. 
		\item The set $\Omega$ has the integrability property. 
	\end{enumerate}
	The correspondence $H\leftrightarrow U(t)$ is given by the Generalized Stone Theorem,
	$$	U(t_1,t_2,\dots,t_d)=\exp\left(2\pi i(t_1H_1+t_2H_2+\dots +t_d H_d)\right)$$ for $(t_1,t_2,\dots,t_d)\in\brr^d)$.
	
\end{theorem}

It is important to note here, that Theorem \ref{thai} is true, even for {\it disconnected} domains $\Omega$.

One of the main goals of our paper is to extend these results to arbitrary open sets in $\brr^d$, possibly disconnected, possibly unbounded. In the case of dimension one, detailed analyses can be found: for unions of two finite intervals in \cite{JPT12,JPT13}, for finite unions of finite intervals in \cite{DJ15b, DJ23, CDJ25}, for finite unions of intervals the form $$\Omega=(-\infty,\beta_1)\cup\bigcup_{j=1}^n(\alpha_j,\beta_j)\cup (\alpha_{n+1},\infty)$$ in \cite{JPT15}.

The paper is structured as follows: in Section \ref{sec2} we extend Theorem \ref{thp5} to include the case of disconnected open sets in $\brr^d$ and we obtain a necessary and sufficient condition for the existence of commuting self-adjoint extensions of the partial differential operators $\{D_j\}$ (Theorem \ref{thi2}). Then we characterize arbitrary, possibly disconnected open subsets of $\brr^d$ which are spectral in terms of the existence of unitary groups of local translations (Definition \ref{deflt}), that is groups that act as translations whenever the translation stays inside $\Omega$: $U(t)f(x)=f(x+t)$ if both $x$ and $x+t$ are inside $\Omega$ (Theorem \ref{thi5}).

 In Section \ref{sec3} we focus on the case when $\Omega$ has finite measure and finitely many components. We show that, in this case, the spectral measure of $\{H_j\}$ is atomic, the spectrum is separated (Definition \ref{defsep}) and there is an orthonormal basis for $L^2(\Omega)$ of functions of the form $\sum_{\Gamma} c_\Gamma\chi_\Gamma e_\lambda$ (that is, different constants $c_\Gamma$, on different connected components $\Gamma$ of $\Omega$, but the same exponential function $e_\lambda$) with $\lambda$ in the joint spectrum of $\{H_j\}$ (Theorem \ref{thf1}); $\chi_\Gamma$ denotes the characteristic function of the set $\Gamma$.
 
 In Section \ref{sec4} we concentrate on the one-dimensional case $d=1$, when $\Omega$ is a possibly countable union of intervals $\Omega=\cup_{i\in I}(\alpha_i,\beta_i)$.
 As we show in Theorem~\ref{th4.1}, under the condition that the intervals have a positive lower bound on their lengths, any unitary group with the integrability property---defined as acting as a translation within each interval---undergoes a transition at the endpoints governed by a unitary ``boundary'' matrix \( B \). Specifically, when the point reaches an endpoint, say \( \alpha_i \), it splits into several points \( \beta_k \), \( k \in I \), with transition probabilities given by \( |b_{i,k}|^2 \). Then we find the relation between $\mu$, $m$, $C$ from Theorem \ref{thi2} and the boundary matrix $B$ (Theorem \ref{thca}). As corollaries we obtain that if there are finitely many intervals and just one of them is unbounded, then no self-adjoint extensions exist (Corollary \ref{cor4.4}); also if there are finitely many intervals and two of them are unbounded and the set $\Omega$ is spectral, then $\Omega$ has to be $\brr$ minus finitely many points (Theorem \ref{th4.5}), a result that can be found also in \cite[Corollary 2.7]{JPT15} with a different proof. These results extend some of the conclusions in \cite{JPT15, DJ15b,DJ23,CDJ25} by allowing infinitely many intervals.  
 
 In Section \ref{sec5}, we examine the scenario where the set $\Omega$ tiles $\mathbb{R}^d$ using a discrete subgroup $\mathcal{T}$ of $\mathbb{R}^d$. Theorem \ref{th5.9} establishes that this tiling property is equivalent to the existence of a pair measure for $\Omega$ supported on the dual set $\mathcal{T}^*$ (as defined in Definition \ref{defdu}). Furthermore, Proposition \ref{prut} provides an explicit formula for the corresponding group of local translations $\{U(t)\}$ in this context. The results generalize one of Fuglede's conclusions from the original paper \cite{Fug74} (Theorem \ref{thfula}) where he proves these for {\it full-rank} lattices in $\brr^d$.

\section{Self-adjoint extensions and unitary groups}\label{sec2}

In our first theorem, we follow Pedersen's approach \cite{Ped87} to present a necessary and sufficient condition for the existence of commuting self-adjoint extensions of the partial differential operators $\{D_j\}$. The idea of the proof is: by the Spectral Theorem, if $\{H_j\}$ are these extensions, then there is a transformation $\mathcal F$ that turns them into multiplications by the component functions $\lambda_j$ on $L^2(\brr^d,\mu,m)$, with some multiplicity $m$ and for some measure $\mu$. This is a generalization of the fact that the Fourier transform changes differentiation into multiplication. Due to the particular structure of $\brr^d$,  pointwise, the components of the transformation $\mathcal F$ behave like distributions $e_k(\lambda)$, for fixed $\lambda\in\brr^d$ and $k\leq m(\lambda)$, 
$$(\mathcal F\varphi)_k(\lambda)=(\varphi, e_k(\lambda)), \quad(\varphi\in C_0^\infty(\Omega) ).$$
For a test function $\varphi\in C_0^\infty(\Omega)$ and a distribution $\delta$ on $\Omega$, we use the notation 
$$(\varphi,\delta)=\cj \delta(\varphi)=\cj{\delta(\cj\varphi)},$$
	to be linear in the first variable and conjugate linear in the second variable, so that if $\delta$ is a function and the integration makes sense, then $(\varphi,\delta)=\int_\Omega \varphi\cj \delta\,dx.$
	
	Then, because $\mathcal F$ changes $H_j$ into multiplication by $\lambda_j$, the distribution $e_k(\lambda)$ satisfies a simple differential equation which shows that it must be the exponential functions $e_\lambda$ with possibly different multiplicative constants on each connected component $\Gamma$ of $\Omega$. These can be thought of as generalized eigenfunctions of the operators $H_j$: $H_je_k(\lambda)=\lambda_j e_k(\lambda)$. 

Our improvement on Pedersen's result consists in that we allow $\Omega$ to be disconnected; this forces us to introduce some extra constants $C_{k,\Gamma}(\lambda)$ which depend on the components $\Gamma$ of $\Omega$, for each $\lambda$ in the joint spectrum of $\{H_j\}$. The multiplicity can also be bigger than one, even infinite (when there are infinitely many components).

We begin with some definitions and notations. 
\begin{definition}\label{defi1}
	We denote by $\mathfrak m_d$ or $\mathfrak m$, the Lebesgue measure on $\brr^d$. 
	
	Let  $\Omega$ be an open set in $\mathbb{R}^{d}$. We denote by $\mathcal{C}(\Omega)$= $\{\Gamma$ : $\Gamma$ is a connected component of $\Omega\}$, and by $p=p(\Omega)$  the cardinality of $\mathcal{C}(\Omega)$, possibly infinite $p=\infty$.

For a function $\varphi$ on $\Omega$, and a component $\Gamma$ in $\mathcal{C}(\Omega),$ we denote by $\varphi_{\Gamma}$, the restriction of $\varphi$ to $\Gamma$: 
$\varphi_{\Gamma}(x)$= $ \left\{ \begin{matrix}
	\varphi(x), & x\in \Gamma \\
	0, & \mbox{ otherwise.}
\end{matrix}\right. $

Define $F\varphi (\lambda)=(\hat{\varphi}_{\Gamma}(\lambda))_{\Gamma\in \mathcal{C}(\Omega)}$  , where $\hat f$ is the Fourier transform of the function $f$ on $\brr^d$,
$$\hat f(\lambda)=\int_{\brr^d} f(x)e^{-2\pi i \lambda\cdot x}\,dx,\quad(\lambda\in\brr^d).$$

For a cardinality $m\in\{0,1,2,\dots,\infty\}$, we denote by $l^2(m)$ the Hilbert space $\{0\}$ if $m=0$, $\bc^m$ if $0<m<\infty$, and $l^2(\bn)$ if $m=\infty$.

A {\it multiplicity function} on $\brr^d$ is a Borel measurable function $m:\brr^d\to \bn\cup\{0,\infty\}$. 

For two multiplicity functions $m$ and $n$ on $\brr^d$, a {\it matrix field} is a measurable function $C$ on $\brr^d$ such that $C(\lambda)$ is an $m(\lambda)\times n(\lambda)$ matrix with entries in $\bc$ for all $\lambda\in\brr^d$. 

For a positive Radon measure $\mu$ on $\mathbb{R}^{d}$ and a multiplicity function $m$ on $\mathbb{R}^{d}$, define $L^{2}(\mu, m)$ to be the direct integral of the Hilbert spaces $l^2(m(\lambda))$, $\lambda\in\brr^d$ with respect to the measure $\mu$. The elements of $L^2(\mu,m)$ are functions $f$ defined on $\brr^d$ with $f(\lambda)\in l^2(m(\lambda))$, for all $\lambda\in\brr^d$, such that each component $\lambda\to f(\lambda)_i$ (in the canonical basis) is measurable, and such that $\int\|f(\lambda)\|^2\,d\mu(\lambda)<\infty$; 
the inner product is defined by $$\langle f,g\rangle_{L^{2}(\mu, m)}=\int_{\mathbb{R}^{d}}\langle f(\lambda),g(\lambda)\rangle_{l^{2}( m(\lambda))}d\mu(\lambda)
=\int_{\brr^d}\sum_{i=1}^{m(\lambda)}f(\lambda)_i\cj{g(\lambda)_i}\,d\mu(\lambda).$$

\end{definition}

\begin{definition}
	\label{defmu}
	Let $\mu$ be a Radon measure and $m$ a multiplicity function on $\brr^d$. Let $f$ be a complex valued Borel measurable function on $\brr^d$. We define the multiplication operator 
	$M_f$ on $L^2(\mu,m)$ by $M_f \varphi(\lambda)=f(\lambda)\varphi(\lambda)$ with domain 
	
	$$\mathscr D(M_f)=\left\{ \varphi\in L^2(\mu,m) :  f\varphi \in L^2(\mu,m)
	\right\}.$$
\end{definition}
The next proposition is well known and easy to prove, it follows directly from the definition of the adjoint of an unbounded operator.
\begin{proposition}\label{prmu}
	The adjoint of the operator $M_f$ is the operator $M_{\cj f}$ with domain $\mathscr{D}(M_{\cj f})=\mathscr{D}(M_f)$. In particular, if $f$ is real valued, then the operator $M_f$ is self-adjoint. 
\end{proposition}

With these notations and definitions we can state our necessary and sufficient condition for the existence of commuting self-adjoint extensions of the partial differential operators $\{D_j\}$ on $\Omega$.

	\begin{theorem}\label{thi2}
	Let $\Omega$ be an open set in $\mathbb{R}^{d}.$ 
	\begin{enumerate}
		\item Suppose that the operators $D_{j}=\frac{1}{2\pi i}\frac{\partial}{\partial x_{j}}$ on $C_{0}^{\infty}(\Omega)$, $j=1,2,\dots,d$, have commuting self-adjoint extensions $H_{j}.$ Then there exists a positive Radon measure $\mu$ on $\mathbb{R}^{d}$, a multiplicity function $m$ on $\brr^d$  and a matrix field $C$ on $\mathbb{R}^{d}$ with $C(\lambda)$ of dimension $m(\lambda)\times p(\Omega)$ such that the map $\mathcal{F}_{C}: C_{0}^{\infty}(\Omega)\to L^{2}(\mu, m)$ defined by
		\begin{equation}\label{eqi2.0}
			\mathcal{F}_{C}\varphi (\lambda)=C(\lambda)F\varphi(\lambda)=\left(\left(\varphi,\sum_{\Gamma\in \mathcal C(\Omega)}\bar C_{k,\Gamma}(\lambda)\chi_\Gamma e_\lambda\right)\right)_{k=1,\dots,m(\lambda)}, (\lambda\in\mathbb{R}^{d}),
		\end{equation}   extends to an isometric isomorphism from $L^{2}(\Omega)$ to $L^{2}(\mu, m)$.
		
		The transform $\mathcal F_C$, changes the operators $H_j$ into multiplication operators: 
		\begin{equation}
			\label{eqi2.1}
			H_{j}=\mathcal{F}_{C}^{-1}M_{\lambda_{j}}\mathcal {F}_{C}, \quad(j=1,2,\dots,d).
		\end{equation}

	Moreover for $\mu$-a.e. $\lambda$, the subspace $$\mathcal R(\lambda)=\{\mathcal F_C\varphi(\lambda) : \varphi\in C_0^\infty(\Omega)\}$$ is dense in $l^2(m(\lambda))$. In particular $m(\lambda)\leq p(\Omega)$ for $\mu$-a.e. $\lambda$, and if $m(\lambda)<\infty$ then the matrix $C(\lambda)$ has full rank.

		\item Conversely, let $\mu$ be a positive Radon measure on $\brr^d$, let $m$ be a multiplicity function on $\brr^d$ and let $C$ be matrix field on $\brr^d$, with $C(\lambda)$ of dimensions $m(\lambda)\times p(\Omega)$. If the map $\mathcal{F}_{C}: C_{0}^{\infty}(\Omega)\to L^{2}(\mu, m)$ defined by $$\mathcal{F}_{C}\varphi (\lambda)=C(\lambda)F\varphi(\lambda),\,\quad (\lambda\in\mathbb{R}^{d}),$$  extends to an isometric isomorphism from $L^{2}(\Omega)$ to $L^{2}(\mu, m)$, then  
		
		$$H_{j}=\mathcal{F}_{C}^{-1}M_{\lambda_{j}}\mathcal {F}_{C}, \quad(j=1,2,\dots,d)$$
		define commuting self-adjoint extensions of the differential operators $D_j$. 
	\end{enumerate}
	
\end{theorem}

\begin{proof}

 If the differential operators $D_{j}$ have commuting self-adjoint extensions $H_{j}$, using the Complete Spectral Theorem (see e.g., \cite[page 205, page 333]{Mau67},\cite[Chapter 6]{Nel69},\cite{vN49},\cite[Theorem A.8]{CDJ25}), we can diagonalize the partial differential operators $H_{j}$, more precisely, transform them into multiplication operators. So, there exists a Radon measure $\mu$ on $\mathbb{R}^{d}$ and a measurable field of Hilbert spaces $\hat{\mathfrak{H}}(\lambda)$, and an isometric isomorphism $\mathcal{F}:L^{2}(\Omega)\rightarrow\hat{\mathfrak{H}}=\int^\oplus\hat{\mathfrak{H}}(\lambda)d\mu(\lambda),$ such that 
	\begin{equation}
		\mathcal{F}(H_{j}f)(\lambda)=\lambda_{j}(\mathcal{F}f)(\lambda), \label{eq1}
	\end{equation}
	for $\lambda=(\lambda_{1},\lambda_{2},\dots ,\lambda_{d})\in \supp\mu$, and $f$ in the domain $\mathscr{D}(H_{j})$ of $H_{j}.$ 

Let 
$$m(\lambda):=\dim \hat{\mathfrak H}(\lambda),\quad( \lambda\in\brr^d).$$

	Also, there exists a sequence of measurable vector fields $\{y_{k}\}$ such that $\{y_{k}(\lambda):k=1,2,\dots,m(\lambda)\}$ is an orthonormal basis for $\hat{\mathfrak{H}}(\lambda),$ $y_{k}(\lambda)=0$ for $k>m(\lambda).$

	Take a sequence $\{\Omega_{j}\}$ of bounded open subsets of $\Omega$ with $\bar{\Omega}_{j}\subset\Omega_{j+1},j=1,2,\dots,$ and $\cup_{j}\Omega_{j}=\Omega$. The inclusion of the pre-Hilbert spaces $C_{0}^{\infty}(\Omega_{j})$ with the inner product from the Sobolev space $H^{m}(\Omega_{j}),$ with $m$ large enough, into $L^{2}(\Omega),$ is a Hilbert-Schmidt operator (see \cite[Theorem 2, page 337]{Mau67}, \cite[Theorem A.14]{CDJ25}). The space $\Phi=C^{\infty}_{0}(\Omega)$ is an inductive limit of the spaces $C^{\infty}_{0}(\Omega_{j}),$ and therefore we can use the the stronger version of the Spectral Theorem (see \cite[The Fundamental Theorem, page 334]{Mau67}, \cite[Theorem A.15]{CDJ25}). It tells us that, for $\mu-$almost every $\lambda,$ there is a distribution $e_{k}(\lambda)$ on $\Omega$ such that 
	
	\begin{equation}
		(\mathcal{F}\varphi)_{k}(\lambda):=\langle\mathcal{F}\varphi(\lambda),y_{k}(\lambda)\rangle =(\varphi, e_{k}(\lambda)),\text{  ($\varphi\in C_{0}^{\infty}(\Omega),k=1,2,\dots,m(\lambda)$)}.\label{eq2}
	\end{equation}

 Fix $\lambda\in\brr^d$. For $\varphi \in C_{0}^{\infty}(\Omega)$, $H_j\varphi=D_j\varphi=\frac{1}{2\pi i}\frac{\partial\varphi}{\partial x_{j}}\in C_0^\infty(\Omega)$, and therefore,
	$$(\mathcal{F}(H_{j}\varphi))_{k}(\lambda)=(\frac{1}{2\pi i}\frac{\partial\varphi}{\partial x_{i}}, e_{k}(\lambda))=(\varphi,\frac{1}{2\pi i}\frac{\partial e_{k}(\lambda)}{\partial x_{j}}).$$
	Also, from \eqref{eq1}, 
	\begin{equation*}
		(\mathcal{F}(H_{j}\varphi))_{k}(\lambda)=\lambda_{j}(\mathcal{F}\varphi)_{k}(\lambda)=\lambda_{j}(\varphi,e_{k}(\lambda))=(\varphi,\lambda_{j}e_{k}(\lambda)).
	\end{equation*}
			This means that $\frac{1}{2\pi i}\frac{\partial e_{k}(\lambda)}{\partial x_{j}}=\lambda_{j}e_{k}(\lambda),$ or, equivalently, $\frac{\partial}{\partial x_{j}}(e_{-\lambda} e_{k}(\lambda))=0,$ for $j=1,2,\dots ,d.$ So, $e_{- \lambda} e_{k}(\lambda)$ is a constant $\alpha_{k,\Gamma}(\lambda)$ on each component, and therefore $e_{k}(\lambda)=\alpha_{k,\Gamma}(\lambda)e_{ \lambda}$, on $\Gamma$, for $k=1,2,\dots,m(\lambda)$, $\Gamma\in \mathcal C(\Omega)$.
	Hence, the distribution $e_k(\lambda)$ is 
	$$e_{k}(\lambda)=\left(\sum_{\Gamma\in\mathcal C(\Omega)}\alpha_{k,\Gamma}(\lambda)\chi_{\Gamma}\right)e_{\lambda}$$
	Substituting in \eqref{eq2} we get $$\langle\mathcal{F}\varphi(\lambda),y_{k}(\lambda)\rangle =(\varphi, e_{k}(\lambda))=(\varphi,(\sum_{\Gamma}\alpha_{k,\Gamma}(\lambda)\chi_{\Gamma})e_{\lambda})$$$$=\sum_{\Gamma}\bar{\alpha}_{k,\Gamma}(\lambda)(\varphi_{\Gamma},e_{\lambda})=\sum_{\Gamma}\bar{\alpha}_{k,\Gamma}(\lambda)\hat{\varphi}_\Gamma(\lambda).$$
	Let $C_{k,\Gamma}(\lambda):=\bar{\alpha}_{k,\Gamma}(\lambda)$, for $k=1,2,\dots, m(\lambda)$, and $\Gamma\in \mathcal{C}(\Omega).$

	Then $$\langle \mathcal{F}\varphi(\lambda),y_{k}(\lambda)\rangle = \sum_{\Gamma} C_{k,\Gamma}(\lambda)\hat{\varphi}_\Gamma(\lambda)$$$$=\text{$k^{th}$ component of }C(\lambda)F\varphi(\lambda)=(\mathcal{F}_{C} \varphi(\lambda))_k. $$

Define for $\lambda\in\brr^d$, $\Psi(\lambda):\hat{\mathfrak{H}}(\lambda)\to l^2(m(\lambda))$, 
$$\Psi(\lambda)(v)=\left(\ip{v}{y_k(\lambda)}_{\hat{\mathfrak{H}}(\lambda)}\right)_{k=1,\dots,m(\lambda)},\quad(v\in\hat{\mathfrak{H}}(\lambda)),$$
and $\Psi:\hat{\mathfrak{H}}\to L^2(\mu,m)$, 
$$\Psi(f)(\lambda)=\Psi(\lambda)f(\lambda),\quad (f\in \hat{\mathfrak{H}},\lambda\in\brr^d).$$
Then $\Psi$ is an isometric isomorphism, and for $\varphi\in C_0^\infty(\Omega)$,  
$$\Psi(\mathcal F\varphi)(\lambda)=\Psi(\lambda)\mathcal F\varphi(\lambda)=\left(\ip{\mathcal F\varphi(\lambda)}{y_k(\lambda)}\right)_k=\left((\mathcal F_C\varphi(\lambda))_k\right)_k=\mathcal F_C\varphi(\lambda).$$
Thus
$$\Psi\circ\mathcal F=\mathcal F_C,\mbox{ on } C_0^\infty(\Omega).$$
Since $\Psi$ and $\mathcal F$ are isometric isomorphisms, $\mathcal F_C$ is an isometry on $ C_0^\infty(\Omega)$ that can be extended to an isometric isomorphism on $L^2(\Omega)$.

To check \eqref{eqi2.1}, note that, for $f\in L^\infty(\mu)$, if $M_f^{\hat{\mathfrak{H}}}$ is the multiplication operator on $\hat{\mathfrak{H}}$, 
$$M_f^{\hat{\mathfrak{H}}}\varphi(\lambda)=f(\lambda)\varphi(\lambda),\quad (\varphi\in\hat{\mathfrak{H}},\lambda\in\brr^d),$$
and $M_f$ is the multiplication operator on $L^2(\mu,m)$, then 
$$\Psi M_f^{\hat{\mathfrak{H}}}=M_f\Psi.$$
We can therefore compute
$$\mathcal F_CH_j=\Psi\mathcal FH_j=\Psi M_{\lambda_j}^{\hat{\mathfrak{H}}}\mathcal F=M_{\lambda_j}\Psi\mathcal F=M_{\lambda_j}\mathcal F_C,$$
and so $H_j=\mathcal F_C^{-1}M_{\lambda_j}\mathcal F_C$. 

We show now that for $\mu$-a.e. $\lambda$, the subspace $\mathcal R(\lambda)=\{\mathcal F_C\varphi(\lambda) : \varphi\in C_0^\infty(\Omega)\}$ is dense in $l^2(m(\lambda))$.  Assume, by contradiction, that this is not true. Then there exists a measurable set $E$ with $\mu(E)>0$, and for $\lambda\in E$, vectors $v(\lambda)$ such that $\|v(\lambda)\|_{l^2(m(\lambda))}=1$ and $v(\lambda)\perp \mathcal R(\lambda)$. Let $v(\lambda)=0$ for $\lambda\not\in E$. 

Since $\mathcal F_C$ is onto, there exists a function $f_0\in L^2(\Omega)$ such that $v=\mathcal F_Cf_0$. Then, for $\varphi\in C_0^\infty(\Omega)$, 
$$\ip{\mathcal F_C f_0}{\mathcal F_C\varphi}_{L^2(\mu,m)}=\int_{\brr^d}\ip{\mathcal F_Cf_0(\lambda)}{\mathcal F_C\varphi(\lambda)}_{l^2(m(\lambda))}\,d\mu(\lambda)$$$$=\int_{\brr^d}\ip{v(\lambda)}{\mathcal F_C\varphi(\lambda)}\,d\mu(\lambda)=0.$$
But, $\{\mathcal F_C\varphi : \varphi\in C_0^\infty(\Omega)\}$ is dense in $L^2(\mu,m)$, so $v=\mathcal F_Cf_0=0$, a contradiction. 

In the case when $m(\lambda)<\infty$, since $\mathcal R(\lambda)$ is a dense subspace, it follows that $l^2(m(\lambda))=\mathcal R(\lambda)=\{C(\lambda) F\varphi(\lambda) : \varphi\in C_0^\infty(\Omega)\}$ and therefore the matrix $C(\lambda)$ has full rank. This shows also that $m(\lambda)\leq p(\Omega)$ for $\mu$-a.e. $\lambda$.

	For (ii), it is clear the the operators $M_{\lambda_j}$ commute (in the sense of commuting spectral projections), and therefore so do the operators $H_j$. Therefore, we only need to check that the operator $H_j$ is an extension of the differential operator $D_j$. 
	
	For $\varphi\in C_0^\infty(\Omega)$, we have
	$$\mathcal F_CD_j\varphi(\lambda)=C(\lambda)FD_j\varphi(\lambda)=C(\lambda)\left(\lambda_j\hat\varphi_\Gamma(\lambda)\right)_\Gamma$$$$=\lambda_jC(\lambda)\left(\hat\varphi_\Gamma(\lambda)\right)_\Gamma=M_{\lambda_j}\mathcal F_C\varphi(\lambda).$$
	This shows that $D_j\subset \mathcal F_C^{-1}M_{\lambda_j}\mathcal F_C=H_j$.

\end{proof}

\begin{remark}\label{remi3}
	If $\{H_j\}$ are commuting self-adjoint extensions of the differential operators $\{D_j\}$, then the data $\mu$, $m$ and $C$ in Theorem \ref{thi2} is not unique. However, suppose $\mu'$, $m'$ and $C'$ is another data for which the map $\mathcal F_{C'}: L^2(\Omega)\to L^2(\mu',m')$, $\mathcal F_{C'}\varphi(\lambda)=C'(\lambda)F\varphi(\lambda)$ extends to isometric isomorphism that transforms the operator $H_j$ on $L^2(\Omega)$ into the multiplication operator $M_{\lambda_j}$ on $L^2(\mu',m')$, i.e.,  	$H_{j}=\mathcal{F}_{C'}^{-1}M_{\lambda_{j}}\mathcal {F}_{C'}$. 
	In this case, $\mathcal F_{C'}\mathcal F_C^{-1}$ is an isometric isomorphism that intertwines the multiplication operators $M_{\lambda_j}$ on $L^2(\mu,m)$ and $M_{\lambda_j}$ on $L^2(\mu',m')$. This means (see \cite[Chapter 6]{Nel69}) that $\mu$ and $\mu'$ are mutually absolutely continuous and $m=m'$, $\mu$-a.e.

	Conversely, suppose $\mu$, $m$ and $C$ are as in Theorem \ref{thi2}, $\mu$ and $\mu'$ are mutually absolutely continuous, and $d\mu=\rho\,d\mu'$ for some measurable function $\rho$ $\mu'$-a.e., then the operator $W:L^2(\mu,m)\to L^2(\mu',m)$, $Wf=\sqrt{\rho}f$ is an isometric isomorphism. Define also $C'(\lambda)=\sqrt{\rho(\lambda)}C(\lambda)$ for $\lambda\in\brr^d$. Then $W\mathcal F_C=\mathcal F_{C'}$, and it is easy to check that $\mu'$, $m$ and $C'$ also satisfy the conclusions of Theorem \ref{thi2}.

		Suppose now we have two data $(\mu, m, C)$ and $(\mu,m,C')$ satisfying the conclusions of Theorem \ref{thi2}, with the same measure $\mu$ and multiplicity function $m$, but different matrix fields $C$ and $C'$. Then $\mathcal{F}_{C'}\mathcal{F}_{C}^{-1}$ is an isometric isomorphism from $L^{2}(\mu,m)$ to itself which commutes with all the multiplication operators $M_{\lambda_{j}}$, $j=1,\dots,d.$ This implies that there exists $W(\lambda)$ unitary in $l^2(m(\lambda))$ for $\mu$-a.e. $\lambda\in \mathbb{R}^{d}$ such that $$\mathcal{F}_{C'}\mathcal{F}^{-1}_{C}f(\lambda)=W(\lambda)f(\lambda),\quad( f\in L^{2}(\mu,m)).$$
	That means that, for $\varphi\in C_0^\infty(\Omega)$,  $\mathcal{F}_{C'}\varphi(\lambda)=W(\lambda)\mathcal{F}_{C}\varphi(\lambda)$, so 
	\begin{equation}
		\label{eqcwc}
		C'(\lambda)F\varphi(\lambda)=W(\lambda)C(\lambda)F\varphi(\lambda).
	\end{equation}
	We need the following lemma:

	\begin{lemma}\label{Fonto}
		\label{Fouto}
		For a fixed $\lambda\in \mathbb{R}^{d}$, the map $$C_{0}^{\infty}(\Omega)\ni \varphi \to F\varphi (\lambda)=(\hat{\varphi}_{\Gamma}(\lambda))_{\Gamma\in \mathcal{C}(\Omega)}$$ 
		is onto if the number of connected components of $\Omega$ $p$ is finite, and it has dense range containing all finitely supported sequences if the number of components $p$ is infinite.
	\end{lemma}
	\begin{proof}
		Let $(a_{\Gamma})_{\Gamma\in\mathcal{C}(\Omega)}$ be a vector with finitely many non-zero components. If $a_{\Gamma}\neq  0$, let $\varphi^{(0)}_{\Gamma}$ be a function in $C_0^\infty(\Omega)$ with compact support contained in $\Gamma$, such that $\int_{\mathbb{R}^{d}}\varphi^{(0)}_{\Gamma}(x)\,dx=a_{\Gamma}$. 
		Let $$\varphi(x)= e^{2\pi i\lambda\cdot x}\sum_{a_{\Gamma}\neq 0}\varphi_{\Gamma}^{(0)}(x).$$ 
		Then,
		$$\hat\varphi_{\Gamma}(\lambda)=\int_{\mathbb{R}^{d}}\varphi_{\Gamma}^{(0)}(x)\,dx=a_{\Gamma},\quad(\Gamma\in\mathcal C(\Omega))$$
		and so $F\varphi(\lambda)=(a_\Gamma)_\Gamma.$
	\end{proof}
	Now using \eqref{eqcwc} and Lemma \ref{Fonto}, we get that $C'(\lambda)=W(\lambda)C(\lambda)$.

	Conversely if $C'(\lambda)=W(\lambda)C(\lambda)$ with $W(\lambda)$ unitary in $l^2(m(\lambda))$ for all $\lambda$, then the operator $W:L^{2}(\mu, m)\to L^{2}(\mu,m)$ defined by, $$Wf(\lambda)=W(\lambda)f(\lambda)$$ is an isometric isomorphism, $W\mathcal F_C= \mathcal F_{C'}$ and $W$ commutes with all multiplication operators $M_{\lambda_j}$. Therefore $(\mu,m,C')$ satisfy the conclusions of Theorem \ref{thi2}.
	
\end{remark}

In the next part of this section, we will focus on the unitary group $\{U(t)\}$ associated as in \eqref{equ} to some commuting self-adjoint extensions $\{H_j\}$ of the differential operators $\{D_j\}$. We know from Theorem \ref{thai}, that $\{U(t)\}$ has the integrability property, it acts as translations inside each connected component, at small scales. 

First, we show that the map $\mathcal F_C$ from Theorem \ref{thi2}, transforms this unitary group into multiplications by exponential functions. This is the analogue of the statement: the Fourier transform changes translation into modulation.

\begin{proposition}
	\label{pri4}
	Let $\{H_j\}$ be commuting self-adjoint extensions of the differential operators $\{D_j\}$ and let $\mu$, $m$ and $C$ be as in Theorem \ref{thi2}. Also, let $\{U(t)\}$ be the unitary group representation of $\brr^d$ given in Theorem \ref{thai}. Then $\{U(t)\}$ has the following formula: define the multiplication operator on $L^2(\mu,m)$
	\begin{equation}\label{eqpri3.1}
		\hat U(t)f(\lambda)=e^{2\pi i\lambda\cdot t} f(\lambda),\quad (f\in L^2(\mu,m), \lambda\in\brr^d, t\in \brr^d).
	\end{equation}
	Then 
	$$\mathcal F_C U(t)=\hat U(t)\mathcal F_C,\quad (t\in \brr^d).$$
\end{proposition}

\begin{proof}
	
	Since $\mathcal F_C$ transforms the operators $\{H_j\}$ into the multiplication operators $\{M_{\lambda_j}\}$, it will also transform the exponential of these operators, which is $\{U(t)\}$, into the exponential of the multiplication operators, which is $\{\hat U(t)\}$.
\end{proof}

Next we will present a necessary and sufficient condition for a set $\Omega$ to be spectral (see Definition \ref{defp1}), in terms of some property of the unitary group $\{U(t)\}$. The integrability property is associated to the existence of commuting self-adjoint extensions $\{H_j\}$ and then the group $\{U(t)\}$ is defined by $U(t)=\exp(2\pi it\cdot H)$ as in \eqref{equ}. The integrability property means that $\{U(t)\}$ acts as translation just inside the connected components in $\Omega$, at small scales. To have a spectral set $\Omega$, the unitary group must satisfy a stronger condition which we call the {\it local translation} property which requires $\{U(t)\}$ to act as translation whenever possible, so when both $x$ and $x+t$
are inside $\Omega$, even if in possibly different components. We make this precise is the next definition. 

	\begin{definition}\label{deflt}
	Let $\Omega$ be an open set in $\mathbb{R}^{d}$ and $\{U(t)\}_{t\in\brr^d}$ be a strongly continuous group of unitary operators on $L^{2}(\Omega).$
	We call ${U(t)}$ a group of local translations (on $\Omega$) if for every open subset $V$ of $\Omega$ and any $f\in\mathbb{R}^{d}$ such that $V+t\subset\Omega$, 
	$$(U(t)f)(x)=f(x+t)\text{   for a.e. $x\in V$ and all $ f\in L^{2}(\Omega).$ }$$ 
\end{definition}
\begin{definition}

	If $(\Omega,\mu)$ is a spectral pair (see Definition \ref{defp1}), one can define a unitary representation of $\brr^d$ on $L^2(\Omega)$ as follows: first define $\hat U(t)$ on $L^2(\mu)$ by 
	$$\hat U(t)f(\lambda)=e^{2\pi i\lambda\cdot t}f(\lambda) \quad (f\in L^2(\mu), \lambda\in\brr^d, t\in\brr^d).$$
	Then define $U(t)$ on $L^2(\Omega)$ by
	\begin{equation}
		\label{eqdef1.1}
		\mathscr  F U(t) f=\hat U(t) \mathscr F f, \quad (f\in L^2(\Omega)).
	\end{equation}
\end{definition}

	\begin{theorem}\label{thi5}
	Let $\Omega$ be an open set in $\mathbb{R}^{d}$.
	\begin{enumerate}
		\item If $\Omega$ is a spectral set then $\{U(t)\}_{t\in\brr^d}$ is a group of local translations. 
		\item Conversely, if there exists a group of local translations $\{U(t)\}_{t\in\brr^d}$ on $L^2(\Omega)$, then $\Omega$ is spectral. The measure $\mu$ that makes $(\Omega,\mu)$ a spectral pair can be obtained as follows: Theorem \ref{thai} gives us commuting self-adjoint extensions $\{H_j\}$ of the differential operators $\{D_j\}$. Theorem \ref{thi2} gives us a measure $\mu$, multiplicity function $m$ and matrix field $C$. In this case, the multiplicity function $m$ is $1$ $\mu$-a.e., and the matrix $C(\lambda)$ of dimension $1\times p(\Omega)$ can be chosen to have all entries equal to 1. Fixing these conditions on $m$ and $C$ gives us the pair measure $\mu$.
	\end{enumerate}
	
\end{theorem}

	\begin{proof} For (i), let $V$ be open subset of $\Omega$ and $t\in\mathbb{R}^{d}$ such that $V+t\in \Omega.$  
	Let $\varphi\in C_{0}^{\infty}(\Omega)$ be supported on $V+t$. We prove that $$(U(t)\varphi)(x)=(T(t)\varphi)(x)=\varphi(x+t),\text{       for all $x\in V$. }$$
	
	We have, for $\lambda\in \mathbb{R}^{d}:$ 
	$$(\mathscr{F}U(t)\varphi)(\lambda)=e^{2\pi i t\cdot \lambda}\hat{\varphi}(\lambda)
$$$$=e^{2\pi it\cdot\lambda}\int_{V+t}\varphi(x)e^{-2\pi ix\cdot\lambda}dx=\int_{V+t}\varphi(x)e^{-2\pi i(x-t)\cdot\lambda}dx$$
	$$=\int_{V}\varphi(x'+t)e^{-2\pi ix'\cdot\lambda}dx'\text{  (substitution $x'=x-t$)}$$
	$$=\int_{\Omega}(T(t)\varphi)(x) e^{-2\pi ix\cdot\lambda}dx =\mathscr{F}(T(t)\varphi)(\lambda).$$
	Since $\mathscr{F}$ is unitary, we obtain $U(t)\varphi=T(t)\varphi.$

	Now, take $f\in L^{2}(\Omega)$, and let $\varphi\in C_{0}^{\infty}(\Omega)$ with support in $V.$ It follows that, for $x\in V+t$ $$(U(-t)\varphi)(x)=(T(-t)\varphi)(x)=\varphi(x-t)$$ and $T(-t)\varphi$ is supported on $V+t.$ 
	Then,
	$$\int_{\Omega}(U(t)f)(x)\bar{\varphi}(x)dx=\langle U(t)f,\varphi\rangle =\langle f, U(-t)\varphi \rangle=\langle f, T(-t)\varphi \rangle$$
	$$=\int_{V+t}f(x)\overline{T(-t)\varphi}(x)dx=\int_{V+t}f(x)\bar{\varphi}(x-t)dx$$$$=\int_{V}f(x+t)\bar{\varphi}(x)dx=\int_{\Omega}f(x+t)\bar{\varphi}(x)dx.$$
	
	Thus, 
	
	$$\int_{\Omega}(U(t)f)(x)\bar{\varphi}(x)dx=\int_{\Omega}f(x+t)\bar{\varphi}(x)dx$$
	for any $\varphi\in C_{0}^{\infty}(\Omega)$ supported on $V.$

	Therefore, $(U(t)f)(x)=f(x+t),$ for a.e. $x\in V$. This means that {$U(t)$} is a group of local translations.

	Conversely, if $\{U(t)\}$ a group of local translation then $\Omega$ has the integrability property and, with Theorem \ref{thai}, there exists $\mu,m$ and $C$ as in Theorem \ref{thi2}.

	Let $V$ be an open subset of $\Omega$ such that $V$ is contained in the component $\Gamma_{i}$ and $V+t$ is contained in the component $\Gamma_{j}.$
	
	Let $\varphi\in C_{0}^{\infty}(\Omega)$ be supported on $V+t$. Then $(U(t)\varphi)(x)=\varphi(x+t)$ is zero outside  $V.$ This follows from the next easy lemma:
	
	\begin{lemma}\label{lem2.8}
		Suppose $V$ is a subset of $\Omega$, $t\in\brr^d$ with $V+t\subset \Omega$. Suppose $f\in L^2(\Omega)$ is zero outside $V+t$, $U(t)$ is unitary on $L^2(\Omega)$, and $U(t)f(x)=f(x+t)$, for all $x\in V$. Then $U(t)f$ is zero outside $V$. 
	\end{lemma}
	
	\begin{proof}
		We have 
		
		$$\int_{\Omega\setminus V} |U(t)f(x)|^2\,dx=\int_\Omega|U(t)f(x)|^2\,dx -\int_{V}|U(t)f(x)|^2\,dx$$$$=\int_\Omega |f(x)|^2\,dx-\int_{V} |f(x+t)|^2\,dx
		=\int_\Omega |f(x)|^2\,dx-\int_{V+t}|f(x)|^2\,dx$$$$=\int_{\Omega\setminus( V+t)}|f(x)|^2\,dx=0.$$
		Thus, $U(t)f$ is supported on $V$.
	\end{proof}
	
	We have, for $\lambda\in \mathbb{R}^{d}:$
	$$(\mathcal{F}_{C}U(t)\varphi)(\lambda)=C(\lambda)((\widehat{U(t)\varphi})_{\Gamma_{k}}(\lambda))_{k}$$
	$$=C(\lambda)\left(\begin{matrix}
		0 \\ \vdots \\ \widehat{(T(f)\varphi)_{\Gamma_{i}}}(\lambda)\\ \vdots \\ 0
	\end{matrix}\right)=e^{2\pi i t\cdot \lambda} \hat{\varphi}(\lambda)\cdot C_{i}(\lambda)$$
	where $C_{i}(\lambda)$ is the $i$-th column of $C(\lambda)$.\\
	On other hand, 
	$$(\mathcal{F}_{C}\varphi)(\lambda)=C(\lambda)(\widehat{\varphi}_{\Gamma_{k}}(\lambda))_{k}$$
	$$=C(\lambda)\left(\begin{matrix}
		0 \\ \vdots \\ \hat{\varphi}_{\Gamma_{j}}(\lambda)\\ \vdots \\ 0
	\end{matrix}\right) =\hat{\varphi}(\lambda)C_{j}(\lambda)$$
	But, 
	$$(\mathcal{F}_{C}U(t)\varphi)(\lambda)=e^{2\pi it\cdot\lambda}\mathcal{F}_{C}\varphi(\lambda).$$
	From the relations obtained above,
	$$e^{2\pi it\cdot\lambda}\hat{\varphi}(\lambda)C_{i}(\lambda)=e^{2\pi it\cdot\lambda}\hat{\varphi}(\lambda)C_{j}(\lambda)$$
	Since $\varphi\in C_{0}^{\infty}(\Omega)$ is arbitrary, we get $C_{i}(\lambda)=C_{j}(\lambda).$ Thus all the columns of $C(\lambda)$ are the same. However, the matrix $C(\lambda)$ has dense range, and therefore $m(\lambda)=1$, for $\mu$-a.e. $\lambda$. In addition, since $C(\lambda)$ has dimensions $m(\lambda)\times p(\Omega)$, all the entries of $C(\lambda)$ are the same, and we denote the value by $c(\lambda)\neq 0$, and therefore $C(\lambda)=c(\lambda)\cdot \mathbf 1$, where $\mathbf 1$ is the $1\times p(\Omega)$ matrix with all entries equal to 1. We have then, for $\varphi\in C_0^\infty(\Omega)$ and $\lambda\in\brr^d$, 
	$$\mathcal F_C\varphi(\lambda)=c(\lambda)\mathbf 1 F\varphi(\lambda)=c(\lambda)\sum_\Gamma\hat\varphi_\Gamma(\lambda)=c(\lambda)\hat\varphi(\lambda).$$
	Since $\mathcal F_C$ is unitary,
	$$\int_\Omega|\varphi(x)|^2\,dx=\|\mathcal F_C\varphi\|^2_{L^2(\mu,m)}=\int|c(\lambda)|^2 |\hat\varphi(\lambda)|^2\,d\mu.$$
	Taking, $d\mu'(\lambda)=|c(\lambda)|^2\,d\mu(\lambda)$, we see that $(\Omega,\mu')$ is a spectral pair. Note also that $\mu'$, $m$ and $C':=\mathbf 1$ can be used as the data in Theorem \ref{thi2}.
	
\end{proof}

\section{Finite measure}\label{sec3}
In this section we focus on the case when $\Omega$ has finite measure and finitely many components. We prove in Theorem \ref{thf1} that, in this case, the spectral measure associated to the commuting self-adjoint extensions $\{H_j\}$ is atomic, supported on a separated set (Definition \ref{defsep}), and so is the associated measure $\mu$ from Theorem \ref{thi2}. The eigenspaces $E(\lambda)$ have a basis of functions of the form $\sum_{\Gamma\in \mathcal C(\Omega)} c_\Gamma \chi_\Gamma e_\lambda$: different constants $c_\Gamma$ for different components $\Gamma$, but the same exponential function $e_\lambda$. Moreover, the constants $c_\Gamma$ are obtained from linear combinations of the conjugate of the rows of the matrix $C(\lambda)$. Putting together all these functions for all $\lambda$ in the spectrum, one obtains an orthonormal basis for $L^2(\Omega)$.
\begin{definition}
	\label{defsep}
	A subset $\Lambda$ of $\brr^d$ is called {\it separated} if there exists a constant $C>0$ such that $\|\lambda-\lambda'\|\geq C$ for all $\lambda,\lambda'\in\Lambda$, $\lambda\neq\lambda'$.
\end{definition}

\begin{theorem}\label{thf1}
	Assume that $\mathfrak m(\Omega)<\infty$ and $\Omega$ has finitely many components. Suppose that $\{H_{j}\}$ are commuting self-adjoint extensions of the operators $\{D_{j}\}$ and let $\mu,m$ and $C$ be as in Theorem \ref{thi2}.
	
		The joint spectral measure $E$ is atomic with $\Lambda:=\supp E=\supp\mu,$ and the support $\Lambda$ is separated. 
	
	 Let 
	\begin{equation}
		\label{eqf1.1}
		E_{\lambda,k}:=\sum_{\Gamma}\bar{C}_{k,\Gamma}(\lambda)\chi_{\Gamma}e_\lambda,\quad(\lambda\in\supp\mu,k=1,\dots,m(\lambda)).
	\end{equation}

	For $\lambda\in\Lambda$, the subspace $E(\lambda)$ is $$E(\lambda)=\operatorname*{span}\{E_{\lambda,k} : k=1,\dots,m(\lambda)\}$$
	and the functions $\{E_{\lambda,k} : k=1,\dots, m(\lambda)\}$ form a basis for $E(\lambda)$. In particular $E_{\lambda,k}$ is in the domain $\mathscr D(H)$.

	For $\varphi\in L^{2}(\Omega),$
	\begin{equation}
		\label{eq11.1}
		\mathcal{F}_{C}\varphi(\lambda)=\left(\ip{ \varphi}{E_{\lambda,k}}\right)_{k=1,\dots,m(\lambda)},
	\end{equation}
	$$\mathcal{F}_{C}H_{j}\mathcal{F}^{-1}_{C}=M_{\lambda_{j}}.$$
	$M_{\lambda_{j}}$ is the multiplication operator on $L^{2}(\mu,m)$ and since $\mu$ is atomic, it can be regarded as a diagonal matrix with entries $\lambda_{j}I_{m(\lambda)}$ at $\lambda.$

	The spectral projection is 
	$$E(\lambda)=\mathcal{F}_{C}M_{\chi_{\{\lambda\}}}\mathcal{F}_{C}^{-1}, $$
	where $M_{\chi_{\{\lambda\}}}$ is the diagonal matrix with entries $\lambda_j\cdot I_{m(\lambda)}$ at $\lambda$ and $0$ everywhere else, or the operator of multiplication by the characteristic function of the point $\lambda$.  
	
	Also the functions $E_{\lambda,k}$ are in the domains $\mathscr{D}(H_j)$, and are eigenvectors for the operators $H_j$ and for the unitary group $\{U(t)\}$ defined by \eqref{equ}: 
	$$H_j E_{\lambda,k}=\lambda_j E_{\lambda,k},\mbox{ and } U(t)E_{\lambda,k}=e^{2\pi i\lambda\cdot t} E_{\lambda,k}.$$

	We can pick some constants $\alpha_{k,\Gamma}(\lambda)$ such that, for each $\lambda\in \supp\mu$, the set $$\left\{\tilde E_{\lambda,k}:=\sum_{\Gamma}\alpha_{k,\Gamma}\chi_{\Gamma}e^{2\pi i\lambda\cdot t}: k=1,\dots,m(\lambda)\right\}$$ is an orthonormal basis for $E(\lambda)$, and then $$\left\{\tilde E_{\lambda, k}:\lambda\in \supp\mu, k=1,\dots,m(\lambda)\right\}$$ is an orthonormal basis for $L^{2}(\Omega).$

	Conversely, if there exists an orthonormal basis for $L^{2}(\Omega)$ of the form 
	$$\{\tilde E_{k,\lambda}=\sum_{\Gamma}\alpha_{k,\Gamma}(\lambda)\chi_{\Gamma}e^{2\pi i\lambda\cdot x}:\lambda\in\Lambda,k=1,\dots,m(\lambda)\},$$
	for some subset $\Lambda$ in $\brr^d$, then let $\delta_\Lambda$ be the counting measure on $\Lambda$ and  define 
	$$\mathcal{F}:L^{2}(\Omega)\to L^{2}(\delta_{\Lambda},m(\lambda)),\quad 
	\mathcal{F}\varphi(\lambda)=(\langle\varphi,\tilde E_{k,\lambda}\rangle)_{k=1,\dots,m(\lambda)},\quad(\varphi\in L^{2}(\Omega)). $$
	The operator $\mathcal{F}$ is an isometric isomorphism.

	Define $$ H_{j}=\mathcal{F}^{-1}M_{\lambda_{j}}\mathcal{F}.$$
	Then $\{H_{j}\}$ are commuting self-adjoint extensions of the operators $D_{j}.$

\end{theorem}

 We will separate the proof into a series of lemmas. 
	\begin{lemma}
		\label{lemma4}
	Suppose that $\Omega$ has finitely many components and let $\varphi\in L^{1}(\Omega)\cap L^{2}(\Omega),$ then, $$(\mathcal{F}_{C}\varphi)(\lambda)=C(\lambda)(\hat{\varphi}_{\Gamma})_{\Gamma} \text{   for $\mu$-a.e. $\lambda\in \mathbb{R}^{d}.$  }$$
	
	\end{lemma}
	\begin{proof}
	Since $\varphi$ is in $L^{1}(\Omega)$ it follows that $\varphi_{\Gamma}\in L^{1}(\Omega)$ for all components $\Gamma$ and therefore $\hat{\varphi}_{\Gamma}$ is continuous and vanishes at $\infty.$
	Suppose now $\{\varphi_{n}\}$ is a sequence in $C_{0}^{\infty}(\Omega)$ and $\varphi_{n}\to \varphi$ in $L^{1}(\Omega)\cap L^{2}(\Omega)$. Then the same is true for $\varphi_{n,\Gamma}\to \varphi_{\Gamma}$.

	This implies that 
	$$|\hat{\varphi}_{n,\Gamma}(\lambda)-\hat{\varphi}_{\Gamma}(\lambda)|=|\int_{\mathbb{R}^{d}}({\varphi}_{n,\Gamma}(x)-{\varphi}_{\Gamma}(x))e^{-2\pi i\lambda\cdot x}dx|$$$$
	\leq\int_{\brr^d} |{\varphi}_{n,\Gamma}(x)-{\varphi}_{\Gamma}(x)|dx=\|{\varphi}_{n,\Gamma}-{\varphi}_{\Gamma}\|_{1}.$$
	Thus  $\hat{\varphi}_{n,\Gamma}$ converges uniformly to $\hat{\varphi}_{\Gamma}$ on $\mathbb{R}^{d}$. 
	Then, since $$\mathcal{F}_{C}\varphi=\lim_{n\to\infty} \mathcal{F}_{C}\varphi_{n}$$ in $L^{2}(\mu,m)$, there exists a subsequence which we label by the same notation $\{\mathcal{F}_{C}\varphi_{n}\}$ that converges to $\mathcal{F}_{C}\varphi$ pointwise $\mu$-a.e. on $\mathbb{R}^{d}.$

	We have for $\mu$-a.e. $\lambda\in\mathbb{R}^{d}$,
	$$\mathcal{F}_{C}\varphi(\lambda)=\lim_{n\to\infty}\mathcal{F}_{C}\varphi_{n}(\lambda)=\lim_{n\to\infty}C(\lambda)(\hat{\varphi}_{n,\Gamma}(\lambda))_{\Gamma}$$
	$$=C(\lambda)(\hat{\varphi}_{\Gamma}(\lambda))_{\Gamma}\text{ (since $\hat{\varphi}_{n,\Gamma}\to\hat{\varphi}_{\Gamma} $ uniformly). }$$
	
	\end{proof}
	
	\begin{lemma}
	\label{lemma5}
	Suppose $\mathfrak m(\Omega)<\infty$ and $\Omega$ has finitely many components, and let $\varphi\in L^{2}(\Omega),t\in \mathbb{R}^{d}.$\\
	Then, for $\mu$-a.e. $\lambda\in \mathbb{R}^{d}$ 
	$$C(\lambda)(\widehat{(U(t)\varphi)_{\Gamma}}(\lambda))_{\Gamma}=e^{2\pi i\lambda\cdot t}C(\lambda)(\hat{\varphi}_{\Gamma}(\lambda))_{\Gamma}.$$
	\end{lemma}
	\begin{proof}
	Since $\mathfrak m(\Omega)<\infty$, we have that $\varphi,U(t)\varphi\in L^1(\Omega)\cap L^2(\Omega)$.	With Lemma \ref{lemma4}, we know that 
	$$\mathcal{F}_{C}(U(t)\varphi)(\lambda)=C(\lambda)(\widehat{(U(t)\varphi)_{\Gamma}}(\lambda))_{\Gamma}$$
	On the other hand, with Proposition \ref{pri4},
	$$\mathcal{F}_{C}(U(f)\varphi)(\lambda)=e^{2\pi i\lambda\cdot t}\mathcal{F}_{C}\varphi(\lambda)=e^{2\pi i\lambda\cdot t}C(\lambda)(\hat{\varphi}_{\Gamma}(\lambda))_{\Gamma}.$$
	Hence the conclusion follows. 
	
	\end{proof}
	
	\begin{lemma}
	\label{lemma6}
	Assume that $\mathfrak m(\Omega)<\infty$ and $\Omega$ has finitely many components. Then for $t\in\brr^d$ and for $\mu$-a.e. $\lambda\in\brr^d$,
	$$U(t)(E_{\lambda,k})=e^{2\pi i\lambda\cdot t}E_{\lambda,k},\text{  for all $k=1,\dots,m(\lambda)$.}$$
	\end{lemma}
	\begin{proof}
	Fix $t\in\brr^d$. We have for $\varphi\in C_{0}^{\infty}(\Omega)$,
	$$\ip{ U(t)\sum_{\Gamma}\bar{C}_{k,\Gamma}(\lambda)\chi_{\Gamma}e_{\lambda}}{\varphi}=\ip{ \sum_{\Gamma}\bar{C}_{k,\Gamma}(\lambda)\chi_{\Gamma}e_{\lambda}}{U(-t)\varphi}
$$$$	=\sum_{\Gamma}\bar{C}_{k,\Gamma}(\lambda)\ip{ e_{\lambda}}{(U(-t)\varphi)_{\Gamma}}$$$$=\sum_{\Gamma}\bar{C}_{k,\Gamma}(\lambda)\cj{\widehat{(U(-t)\varphi)_{\Gamma}}(\lambda)}=\sum_{\Gamma}\bar{C}_{k,\Gamma}(\lambda)\overline{e^{-2\pi i\lambda\cdot t}\hat{\varphi}_{\Gamma}(\lambda)}\text{ (by Lemma \ref{lemma5}) }$$
	for $\mu$-a.e. $\lambda$, and $k=1,\dots,m(\lambda).$

	On the other hand
	$$\ip{ e^{2\pi i\lambda\cdot t}\left(\sum_{\Gamma}\bar{C}_{k,\Gamma}(\lambda)\chi_{\Gamma}\right)e_{\lambda}}{\varphi} =e^{2\pi i\lambda\cdot t}\sum_{\Gamma}\bar{C}_{k,\Gamma}(\lambda)\ip{ e_{\lambda}}{\varphi_\Gamma}$$$$=e^{2\pi i\lambda\cdot t}\sum_{\Gamma}\bar{C}_{k,\Gamma}(\lambda)\overline{\hat{\varphi}_{\Gamma}(\lambda)}.$$
	It follows that 
	\begin{equation}
		\label{6}
		\ip{ U(t)\sum_{\Gamma}\bar{C}_{k,\Gamma}(\lambda)\chi_{\Gamma}e_{\lambda}}{\varphi}=\ip{ e^{2\pi i\lambda\cdot t}\left(\sum_{\Gamma}\bar {C}_{k,\Gamma}(\lambda)\chi_{\Gamma}\right)e_{\lambda}}{\varphi},
	\end{equation}
 for $\mu$-a.e. $\lambda$.
	Note that the $\mu$-null set in equation \eqref{6} depends on $\varphi$.

	Pick a countable dense set $\mathfrak{D}$ of $\varphi$'s.
	Then equation \eqref{6} holds for any $\varphi\in \mathfrak{D}$ and $\lambda$ outside a $\mu$-null set.
	Therefore,
	$$U(t)\left(\sum_{\Gamma}\bar{C}_{k,\Gamma}(\lambda)\chi_{\Gamma}e_{\lambda}\right)=e^{2\pi i\lambda\cdot t}\left(\sum_{\Gamma}\bar{C}_{k,\Gamma}(\lambda)\chi_{\Gamma}e_\lambda\right )$$
 for $\lambda$ outside a $\mu$-null set.

	\end{proof}
	
	\begin{lemma}
	For two distinct points $\lambda$ and $\gamma$ outside the $\mu$-null set in Lemma \ref{lemma6}, $E_{\lambda,k}$ is orthogonal to $E_{\gamma,k'}$, 
	where $k=1,\dots,m(\lambda)$ and $k'=1,\dots,m(\gamma).$
	\end{lemma}
	\begin{proof}
	We have, by Lemma \ref{lemma6}, 
	$$A:=\ip{ U(t)E_{\lambda,k}}{E_{\gamma,k'}}=e^{2\pi i\lambda\cdot t}\ip{ E_{\lambda,k}}{E_{\gamma,k'}}.$$
	Also,
	$$A=\ip{ E_{\lambda,k}}{U(-t)E_{\gamma,k'}}=e^{2\pi i\gamma\cdot t}\ip{ E_{\lambda,k}}{E_{\gamma,k'}}.$$
	Thus, since $\lambda\neq \gamma$, we must have $\ip{ E_{\lambda,k}}{E_{\gamma,k'}}=0.$
	\end{proof}
	
	\begin{lemma}
	\label{lemma8}
	For $\lambda$ outside the $\mu-$null set in Lemma \ref{lemma6}, $$\mathcal{F}_{C}(E_{\lambda,k})(\gamma)=0,\text{  for $\mu$-a.e. $\gamma\neq \lambda.$  }$$
	\end{lemma}
	\begin{proof}
	$$\mathcal{F}_{C}(E_{\lambda,k})(\gamma)=C(\gamma)(\widehat{(E_{\lambda,k})_{\Gamma}}(\gamma))_\Gamma=C(\gamma)\left(\bar{C}_{k,\Gamma}(\lambda)\widehat{\chi_{\Gamma}e_{\lambda}}(\gamma)\right)_{\Gamma}$$
	$$=\left(\sum_{\Gamma}C_{k',\Gamma}(\gamma)\bar{C}_{k,\Gamma}(\lambda)\hat{\chi}_{\Gamma}(\gamma-\lambda)\right)_{k'=1,\dots,m(\gamma)}.$$
	On the other hand, for $k'=1,\dots,m(\gamma)$,
	$$0=\ip{ E_{\lambda,k}}{E_{\gamma,k'}}=\sum_{\Gamma}\bar{C}_{k,\Gamma}(\lambda)C_{k',\Gamma}(\gamma)\int_{\Gamma}e_{\lambda}\bar{e_{\gamma}}$$
	$$=\sum_{\Gamma}C_{k',\Gamma}(\gamma)\bar{C}_{k,\Gamma}(\lambda)\hat{\chi}_{\Gamma}(\gamma-\lambda).$$
	Thus, for $\gamma\neq \lambda$,
	$$\mathcal{F}_{C}(E_{\lambda,k})(\gamma)=0.$$
	\end{proof}
	\begin{lemma}
	If $\mathfrak m(\Omega)<\infty$ and $\Omega$ has finitely many components, then the measure $\mu$ is atomic. 
	\end{lemma}
	\begin{proof}
	From Lemma \ref{lemma8}, we have that for $\mu$-a.e. $\lambda$, $\mathcal{F}_{C}E_{\lambda,k}(\gamma)=0$ for $\mu$-a.e. $\gamma\neq\lambda.$ 
	
	Thus $\mathcal{F}_{C}(E_{\lambda,k})=\delta_{\lambda}c$ in $L^{2}(\mu,m)$ for some constant vector $c\neq0$ in $\mathbb{C}^{m(\lambda)}.$ 
	This means that, $$\norm{c}^{2}\mu(\{\lambda\})=\norm{\delta_{\lambda}c}^{2}_{L^{2}(\mu,m)}=\norm{\mathcal{F}_{C}E_{\lambda,k}}^{2}_{L^{2}(\Omega)}>0,$$
	which implies that $\mu(\{\lambda\})>0$ for $\mu$-a.e. $\lambda.$ Therefore $\mu$ is atomic. 
	\end{proof}
	Note also, that since $\mathcal F_C(E_{\lambda,k})=\delta_\lambda c$ and $H_j=\mathcal F_C^{-1}M_{\lambda_j}\mathcal F_{C}$, we get that $E_{\lambda,k}$ is in the eigenspace $E(\lambda)$, which means also that it is in the domain $\mathscr D(H)$.
	\begin{lemma}
	If $\mathfrak m(\Omega)<\infty$ and $\Omega$ has finitely many components, then the support of $\mu$ is separated. 
	\end{lemma}
	\begin{proof}
	Suppose that the support of $\mu$ is not separated and take some distinct $\lambda_{n}\neq \lambda_{n}'$ in the support of $\mu$ with $\lambda_{n}-\lambda_n'\to0.$ Use some constant $a_{n}$ to normalize 
	$$1=\norm{a_{n}E_{\lambda_{n},1}}_{L^2(\Omega)}=(\sum_{\Gamma}|a_{n}|^{2}|C_{1,\Gamma}(\lambda_{n})|^{2}\mathfrak m(\Gamma))^{\frac{1}{2}},$$
	and similarly $1=\norm{a_{n}'E_{\lambda_{n}',1}}_{L^2(\Omega)}$.

	We have 
	$$0=\ip{ a_{n}E_{\lambda_{n},1}}{ a_{n}'E_{\lambda_{n}',1}}=a_{n}\bar{a}_{n}'\sum_{\Gamma}\bar{C}_{1,\Gamma}(\lambda_{n})C_{1,\Gamma}(\lambda_{n}')\hat{\chi}_{\Gamma}(\lambda_{n}'-\lambda_{n}).$$

	For $n$ large $\lambda_{n}'-\lambda_{n}$ is close to zero, and since all the functions $\hat{\chi}_{\Gamma}$ are continuous with $\hat\chi_\Gamma(0)=\mathfrak m(\Gamma)$, we can make sure that 
	\begin{equation}
		\label{eq9}
		|\hat{\chi}_{\Gamma}(\lambda_{n}'-\lambda_{n})- \mathfrak m(\Gamma)|<\epsilon\text{  for $n\geq N$ }
	\end{equation}

	Then,
	$$|\sum_{\Gamma}a_{n}\bar{a}_{n}'\bar{C}_{1,\Gamma}(\lambda_{n})C_{1,\Gamma}(\lambda_{n}')\mathfrak m(\Gamma)|$$$$=|\sum_{\Gamma}a_{n}\bar{a}_{n}'\bar{C}_{1,\Gamma}(\lambda_{n})C_{1,\Gamma}(\lambda_{n}')(\mathfrak m(\Gamma)-\hat{\chi_{\Gamma}}(\lambda_{n}'-\lambda_{n})+\hat{\chi_{\Gamma}}(\lambda_{n}'-\lambda_{n}))|$$
	$$\leq|\sum_{\Gamma}a_{n}\bar{a}_{n}'\bar{C}_{1,\Gamma}(\lambda_{n})C_{1,\Gamma}(\lambda_{n}')(\mathfrak m(\Gamma)-\hat{\chi_{\Gamma}}(\lambda_{n}'-\lambda_{n}))|
	$$$$+|\sum_{\Gamma}a_{n}\bar{a_n}'\bar{C}_{1,\Gamma}(\lambda_{n})C_{1,\Gamma}(\lambda_{n}')\hat{\chi_{\Gamma}}(\lambda_{n}'-\lambda_{n})|$$
	$$\leq \epsilon(\sum_{\Gamma}|a_{n}\bar{C}_{1,\Gamma}(\lambda_{n})|^{2})^{\frac{1}{2}}(\sum_{\Gamma}|a_{n}'\bar{C}_{1,\Gamma}(\lambda_{n}')|^{2})^{\frac{1}{2}}$$
by \eqref{eq9} and the Schwarz inequality.
	But, 
	$$\sum_{\Gamma}|a_{n}\bar{C}_{1,\Gamma}(\lambda_{n})|^{2}=\sum_{\Gamma}|a_{n}\bar{C}_{1,\Gamma}(\lambda_{n})|^{2}\mathfrak m(\Gamma)\frac{1}{\mathfrak m(\Gamma)}$$
	$$\leq \norm{a_{n}E_{\lambda_{n},1}}_{L^2(\Omega)}^{2}\cdot\frac{1}{\min_{\Gamma}\mathfrak m(\Gamma)}=\frac{1}{\min_{\Gamma}\mathfrak m(\Gamma)}=:c.$$
	We denote by $v_{n}$ the vector 
	$$v_{n}=(a_{n}\bar{C}_{1,\Gamma}(\lambda_{n})\sqrt{ \mathfrak m(\Gamma)})_{\Gamma}\in \mathbb{C}^{p(\Omega)},$$
	and similarly for $v_n'$ with $\lambda_n'$. 
	
	We have $|\langle v_{n},v_{n}'\rangle|<\epsilon\cdot c$ for $n\geq N.$		Also, $\norm{v_{n}}^{2}=\norm{a_{n}E_{\lambda_{n},1}}_{L^2(\Omega)}^2=1=\norm{v_n'}^2$ for all $n.$ 
	But, since $\#\Gamma=p(\Omega)<\infty,$ the vectors $v_{n}, v_n'$ are in a finite dimensional space, and by passing to a subsequence, we can assume that $v_{n}\to v$ and $v_n'\to v'$ in $\mathbb{C}^{\#\Gamma}.$

	However, this implies that $|\langle v_{n},v_{n}'\rangle|\to 1$ which contradicts $|\langle v_{n},v_{n}'\rangle|\leq \epsilon\cdot c.$

	\end{proof}

	We finalize here the proof of Theorem \ref{thf1}:

	\begin{proof}[Proof of Theorem \ref{thf1}]
	
	We prove that the functions $\{E_{\lambda,k} : k=1,\dots, m(\lambda)\}$ are linearly independent. By Theorem \ref{thi2}, the matrix $C(\lambda)$ has full rank. Then the row vectors 
	$\{( C_{k,\Gamma}(\lambda))_\Gamma : k=1,\dots, m(\lambda)\}$ are linearly independent. 
	
	Suppose now there are some constants $\alpha_k$, $k=1,\dots,m(\lambda)$ such that $\sum_{k=1}^{m(\lambda)}\alpha_k E_{\lambda,k}=0$. Then 
	$$0=\sum_k\alpha_k\sum_\Gamma \bar C_{k,\Gamma}\chi_\Gamma e_\lambda=
	\sum_\Gamma\left(\sum_k \alpha_k\bar C_{k,\Gamma}\right)\chi_\Gamma e_\lambda.$$
	This implies that $\sum_k \alpha_k \bar C_{k,\Gamma}=0$ for all $\Gamma$ which means that $\sum_k \bar \alpha_k (C_{k,\Gamma})_\Gamma=0$. Since these row vectors are linearly independent, it follows that $\alpha_k=0$ for all $k$.
	
	We know from Theorem \ref{thi2} that $\mathcal F_C^{-1}M_{\lambda_j}\mathcal F_C=H_j$. Therefore $\mathcal F_C$ makes the transfer between the spectral projections and so $E(\lambda)=\mathcal F_C^{-1}M_{\chi_{\{\lambda\}}}\mathcal F_C$. In particular, the dimension of $E(\lambda)$ is the dimension of the range of the projection $M_{\chi_{\{\lambda\}}}$ which is $m(\lambda)$. Thus, the functions $\{E_{\lambda,k}: k=1,\dots,m(\lambda)\}$ form a basis for $E(\lambda)$. This implies also that some linear combinations $\tilde E_{\lambda,k}$ form an {\it orthonormal basis} for $E(\lambda)$. Since $\oplus_{\lambda\in\Lambda} E(\lambda)=I_{L^2(\Omega)}$ it follows that $\{\tilde E_{\lambda,k}: \lambda\in\Lambda, k=1,\dots,m(\lambda)\}$ form an orthonormal basis for $L^2(\Omega)$.
	
	From Lemma \ref{lemma8} we have that $\mathcal F_CE_{\lambda,k}=\delta_\lambda\cdot c$ for some vector $c$ in $m(\lambda)$. Since 
	$M_{\chi_{\{\lambda\}}}(\delta_\lambda c)=\delta_{\lambda} c$ it follows, by taking the inverse transformation $\mathcal F_C^{-1}$, that $E(\lambda) E_{\lambda,k}=E_{\lambda,k}$. This implies also that $E_{\lambda,k}\in\mathscr{D}(H_j)$, $H_j E_{\lambda,k}=\lambda_j E_{\lambda,k}$, and $U(t) E_{\lambda,k}=e^{2\pi i \lambda\cdot t} E_{\lambda,k}$ for all $j=1,\dots,d$. 
	
	We check equation \eqref{eq11.1}. We have, for $\varphi\in C_0^\infty(\Omega)$ and $\lambda\in\brr^d$, 
	$$\mathcal F_C\varphi(\lambda)=\left(\sum_\Gamma C_{k,\Gamma}(\lambda)\hat\varphi_\Gamma(\lambda)\right)_k
	=\left(\sum_\Gamma C_{k,\Gamma}(\lambda)\ip{\varphi_\Gamma}{e_\lambda}\right)_k
	$$$$=\left(\ip{\varphi}{\sum_\Gamma \bar C_{k,\Gamma}(\lambda)\chi_\Gamma e_\lambda}\right)_k
	=\left(\ip{\varphi}{E_{\lambda,k}}\right)_k.$$

		For the converse, repeating a computation from above, we see that, for $\lambda\in\Lambda$, 
	$$\mathcal F\varphi(\lambda)=\left( \ip{\varphi}{\tilde E_{\lambda,k}}\right)_{k=1,\dots,m(\lambda)}=\left(\sum_\Gamma \bar\alpha_{k,\Gamma}(\lambda)\hat\varphi_\Gamma(\lambda)\right)_k$$$$=C(\lambda)F\varphi(\lambda)=\mathcal F_C\varphi(\lambda),$$
	where $C_{k,\Gamma}(\lambda)=\bar\alpha_{k,\Gamma}(\lambda)$. The converse follows from Theorem \ref{thi2}(ii). 
\end{proof}

\section{One dimension}\label{sec4}

In this section we focus on the one-dimensional case $d=1$; $\Omega$ is now a finite or countable union of disjoint, possibly unbounded intervals
$$\Omega=\cup_{i\in I}(\alpha_i,\beta_i).$$

 We know from Theorem \ref{thai} that the unitary group associated to a self-adjoint extension $H$ of the differential operator $D$ has the integrability property (Definition \ref{defp5}), that is, for small values of $t$ and if $x$ and $x+t$ are in the same interval of $\Omega$, then $U(t)$ acts as translation: $U(t)f(x)=f(x+t)$. The question is, what happens when $x+t$ goes through a boundary point $\alpha_i$ or $\beta_i$? In Theorem \ref{th4.1}, we show that when a point $x+t$ transitions through, say, a left-boundary point $\alpha_i$, it has to ``split'' into the right-boundary points $\beta_k$ with probabilities given by the absolute value square of the entries $b_{i,k}$, $k\in I$ of a unitary matrix $B$.  

\begin{theorem}\label{th4.1}
Let $\Omega$ be a finite or countable union of disjoint, possibly unbounded intervals,
$$\Omega=\bigcup_{i\in I}(\alpha_i,\beta_i).$$
Let $\{U(t)\}_{t\in\brr}$ be a strongly continuous unitary group with the integrability property on $L^2(\Omega)$. 
Assume in addition that the lengths of the intervals are bounded below by a positive number:
\begin{equation}
	\label{eq4.1.1}
	l:=\inf_{i\in I}(\beta_i-\alpha_i)>0.
\end{equation}
Then there exists a unitary matrix $B=(b_{i,k})_{i,k\in I,\alpha_i>-\infty,\beta_k<\infty}$, such that for any $\epsilon>0$ with $0<\epsilon<l$, and any $i\in I$ with $\alpha_i>-\infty$, if the function $f\in L^2(\Omega)$ is supported on $[\alpha_i,\alpha_i+\epsilon]$, then $U(\epsilon)f$ is supported on $\cup_{k\in I, \beta_k<\infty}[\beta_k-\epsilon,\beta_k]$ and 
	\begin{equation}
	\label{eq4.2.0}
	U(\epsilon)f=\sum_{k\in I,\beta_k<\infty}b_{i,k}T(\alpha_i+\epsilon-\beta_k)f.
\end{equation}

Similarly if $\beta_k<\infty$ and $f$ is supported on $[\beta_k-\epsilon,\beta_k ]$ then $U(-\epsilon) f$ is supported on $\cup_{j,\alpha_j>-\infty}[\alpha_j,\alpha_j+\epsilon]$ and 
\begin{equation}
	\label{eq4.2.0b}
	U(-\epsilon)f=\sum_{j\in I, \alpha_j>-\infty}\bar b_{j,k} T(\beta_k-\alpha_j-\epsilon)f.
\end{equation}

Conversely, if $B=(b_{i,k})_{i,k\in  I,\alpha_i>-\infty,\beta_k<\infty}$ is a unitary matrix, then there exists a unitary group with the integrability property $\{U(t)\}$ such that the relations \eqref{eq4.2.0} and \eqref{eq4.2.2} are satisfied. 

\end{theorem}

\begin{proof}
We begin with a lemma:	
	
	\begin{lemma}
		\label{lem4.2}
		Let $0<\epsilon<l$. 
		Suppose that $\alpha_i>-\infty$ and that the function $f$ in $L^2(\Omega)$ is supported on $[\alpha_i,\alpha_i+\epsilon]$. Then $U(\epsilon)f$ is supported on $\cup_{k\in I, \beta_k<\infty}[\beta_k-\epsilon,\beta_k]$. 
	\end{lemma}

	\begin{proof}
		Let $k\in I$ be arbitrary. Take $\varphi\in C_0^\infty(\Omega)$ supported on $[\alpha_k,\beta_k-\epsilon]$ (if $\beta_k=\infty$ then $\beta_k-\epsilon=\infty$). Then, since $\{U(t)\}$ has the integrability property, we have, by Lemma \ref{lem2.8}, that $U(-\epsilon)\varphi$ is supported on $[\alpha_k+\epsilon,\beta_k]$. This implies that $U(-\epsilon)\varphi$ is orthogonal to $f$. Then 
		$\varphi=U(\epsilon)U(-\epsilon)\varphi\perp U(\epsilon)f$. Since $\varphi$ is arbitrary, it follows that $U(\epsilon)f$ has to be supported on $\cup_k [\beta_k-\epsilon,\beta_k]$, and for the case when $\beta_k=\infty$, the interval $[\beta_k-\epsilon,\beta_k]$ is empty. 
		
	\end{proof}
	
	Let $0<\epsilon<l$. With Lemma \ref{lem4.2} we know that $U(\epsilon)\chi_{(\alpha_i,\alpha_i+\epsilon)}$ is supported on $\cup_{k,\beta_k<\infty} [\beta_k-\epsilon,\beta_k]$.  We will show that $U(\epsilon)\chi_{(\alpha_i,\alpha_i+\epsilon)}$ is constant on each interval $[\beta_k-\epsilon,\beta_k]$.
	
	Define 
	$$g_n:=U\left(\frac\epsilon{n}\right)\chi_{\left(\alpha_i,\alpha_i+\frac\epsilon{n}\right)},\quad (n\in\bn).$$
	
	We have 
	$$\chi_{(\alpha_i,\alpha_i+\epsilon)}=\sum_{j=0}^{n-1}\chi_{\left(\alpha_i+\frac{j}{n}\epsilon, \alpha_i+\frac{j+1}{n}\epsilon\right)};$$
	\begin{equation}
		\label{eq4.2.2}
			g_1=U(\epsilon)\chi_{(\alpha_i,\alpha_i+\epsilon)}=\sum_{j=0}^{n-1}U(\epsilon)\chi_{\left(\alpha_i+\frac{j}{n}\epsilon, \alpha_i+\frac{j+1}{n}\epsilon\right)}.
	\end{equation}

	We write, using the integrability property of $\{U(t)\}$:
	$$U(\epsilon)\chi_{\left(\alpha_i+\frac{j}{n}\epsilon, \alpha_i+\frac{j+1}{n}\epsilon\right)}=
	U\left(\frac{n-j-1}{n}\epsilon\right)U\left(\frac{\epsilon}{n}\right)U\left(\frac{j}{n}\epsilon\right)\chi_{\left(\alpha_i+\frac{j}{n}\epsilon, \alpha_i+\frac{j+1}{n}\epsilon\right)}$$$$=U\left(\frac{n-j-1}{n}\epsilon\right)U\left(\frac{\epsilon}{n}\right)\chi_{\left(\alpha_i,\alpha_i+\frac\epsilon{n}\right)}
	=U\left(\frac{n-j-1}{n}\epsilon\right)g_n.$$
	Using Lemma \ref{lem4.2}, $g_n$ is supported on $\cup_{k,\beta_k<\infty} [\beta_k-\frac{\epsilon}{n},\beta_k]$; then, with the integrability property 
	$$U\left(\frac{n-j-1}{n}\epsilon\right)g_n=T\left(\frac{n-j-1}{n}\epsilon\right)g_n.$$
	Plugging into \eqref{eq4.2.2}, we obtain that 
	\begin{equation}
		\label{eq4.2.3}
		g_1=\sum_{j=0}^{n-1}T\left(\frac{n-j-1}{n}\epsilon\right)g_n.
	\end{equation}
	Note also that $T\left(\frac{n-j-1}{n}\epsilon\right)g_n$ is supported on $$\cup_{k,\beta_k<\infty}\left[\beta_k-\frac{n-j}{n}\epsilon,\beta_k-\frac{n-j-1}{n}{\epsilon}\right].$$ In particular, these functions are disjointly supported for different $j$'s (up to measure zero). This means that 
	$$g_1(x)=T\left(\frac{n-j-1}{n}\epsilon\right)g_n(x)\mbox{ for }x \in\left[\beta_k-\frac{n-j}{n}\epsilon,\beta_k-\frac{n-j-1}{n}{\epsilon}\right],$$ for any $k\in I$, with $\beta_k<\infty$ and any $j\in\{0,\dots,n-1\}$. Therefore the restriction of $g_1$ to the interval $[\beta_k-\epsilon,\beta_k]$ has period $\frac{\epsilon}{n}$. Since $n$ is arbitrary, the restriction of $g_1$ to the interval $[\beta_k-\epsilon, \beta_k]$ must be some constant $b_{i,k}$.
	
	Moreover, since $\|U(\epsilon)\chi_{(\alpha_i,\alpha_i+\epsilon)}\|^2=\|\chi_{(\alpha_i,\alpha_i+\epsilon)}\|^2$, a simple computation shows that $\sum_{k\in I,\beta_k<\infty}|b_{i,k}|^2=1$. Also, since, for $i\neq i'$, $\chi_{(\alpha_i,\alpha_i+\epsilon)}\perp\chi_{(\alpha_{i'},\alpha_{i'}+\epsilon)}$, it follows that $\sum_{k\in I,\beta_k<\infty}b_{i,k}\cj b_{i',k}=0$.  Thus the matrix $B$ is the matrix of an isometry. 
	
	Consider now a function $f\in L^2(\Omega)$ which is supported on $[\alpha_i,\alpha_i+\epsilon]$. We prove that 
	\begin{equation}
		\label{eq4.2.4}
		U(\epsilon)f=\sum_{k\in I, \beta_k<\infty}b_{i,k}T(\alpha_i+\epsilon-\beta_k)f.
	\end{equation}
	
	It is enough to prove this for the case when $f$ is of the form $f=\chi_{\left(\alpha_i+\frac{j}{n}\epsilon,\alpha_i+\frac{j+1}{n}\epsilon\right)}$. Indeed, if this is true, then by linearity, \eqref{eq4.2.4} is true for piecewise constant functions, and then, by approximation, it is true for any 
	function in $L^2(\Omega)$ supported on $[\alpha_i,\alpha_i+\epsilon]$.
	
	But,
	$$U(\epsilon)\chi_{\left(\alpha_i+\frac{j}{n}\epsilon,\alpha_i+\frac{j+1}{n}\epsilon\right)}=U\left(\frac{n-j-1}{n}\epsilon\right)U\left(\frac\epsilon{n}\right)U\left(\frac jn\epsilon\right)\chi_{\left(\alpha_i+\frac{j}{n}\epsilon,\alpha_i+\frac{j+1}{n}\epsilon\right)}$$$$=U\left(\frac{n-j-1}{n}\epsilon\right)U\left(\frac\epsilon{n}\right)\chi_{(\alpha_i,\alpha_i+\frac{\epsilon}{n})}$$
	$$=U\left(\frac{n-j-1}{n}\epsilon\right)\sum_{k\in I}b_{i,k}\chi_{(\beta_k-\frac\epsilon{n},\beta_k)}=\sum_{k\in I}b_{i,k}\chi_{\left(\beta_k-\frac{n-j}{n}\epsilon,\beta_k-\frac{n-j-1}{n}\epsilon\right)}
	$$$$=\sum_{k\in I}b_{i,k}T(\alpha_i+\epsilon-\beta_k)\chi_{\left(\alpha_i+\frac{j}{n}\epsilon,\alpha_i+\frac{j+1}{n}\epsilon\right)}.$$

	Similar arguments can be used to show that there exists a matrix of an isometry  $C=(c_{k,j})_{k,j\in I, \beta_k<\infty,\alpha_j>-\infty}$,  such that, if $\beta_k<\infty$ and $f$ is supported on $[\beta_k-\epsilon,\beta_k ]$ then $U(-\epsilon) f$ is supported on the union of intervals $\cup_{j,\alpha_j>-\infty}[\alpha_j,\alpha_j+\epsilon]$ and 
	$$U(-\epsilon)f=\sum_{j\in I, \alpha_j>-\infty}c_{k,j} T(\beta_k-\alpha_j-\epsilon)f.$$
	
	Since $U(\epsilon)U(-\epsilon) f=f$ we obtain 
	
$$f=U(\epsilon)\sum_j c_{k,j}T(\beta_k-\alpha_j-\epsilon)f$$$$=\sum_jc_{k,j}\sum_{k'}b_{j,k'} T(\alpha_j+\epsilon-\beta_{k'})T(\beta_k-\alpha_j-\epsilon)f$$
$$=\sum_jc_{k,j}\sum_{k'}b_{j,k'} T(\beta_k-\beta_{k'})f=\sum_{k'}T(\beta_k-\beta_{k'})\sum_jc_{k,j}b_{j,k'}.$$
The function $f$ is supported on $[\beta_k-\epsilon,\beta_k]$ and the function $T(\beta_k-\beta_{k'})f$ is supported on $[\beta_{k'}-\epsilon,\beta_{k'}]$, for any $k'$. This implies that 
$$\sum_{j} c_{k,j}b_{j,k'}=\delta_{k,k'}.$$

This means that $CB=I$, and since $C,B$ are isometries we obtain that $C$ is also onto, so it is unitary, $C=B^*$ and $B$ is unitary as well.

For the converse, let $0<\epsilon$ such that $2\epsilon<l$. 
We will need a lemma:
\begin{lemma}
	\label{leml2}
	Let $f\in L^2(\Omega)$. Then $(f(\alpha_i+u))_{i\in I}\in l^2(I)$ for a.e. $u\in(0,2\epsilon)$. Also, 
	\begin{equation}\label{eql2}
		\sum_{k, \beta_k<\infty}\left|\sum_{i, \alpha_i>-\infty}b_{i,k}f(\alpha_i+u)\right|^2=\sum_{i,\alpha_i>\infty}|f(\alpha_i+u)|^2<\infty.
	\end{equation}
\end{lemma}

\begin{proof}
	We have 
	$$\infty>\sum_{i\in I}\int_{(\alpha_i,\alpha_i+2\epsilon)}|f(x)|^2\,dx=\sum_i\int_{(0,2\epsilon)}|f(\alpha_i+u)|^2\,du$$$$=\int_{(0,2\epsilon)}\sum_i|f(\alpha_i+u)|^2\,du.$$
	Then $\sum_i |f(\alpha_i+u)|^2<\infty$ for a.e $u\in (0,2\epsilon)$.
	
	For \eqref{eql2}, note that the left-hand side is the $l^2$ norm of $B^T(f(\alpha_i+u))_i$. Since $B^T$ is unitary, this is the same as the norm of $(f(\alpha_i+u))_i$.
\end{proof}

For $0\leq t\leq 2\epsilon$, and $f\in L^2(\Omega)$ define 
$$U(t)f(x)=	f(x+t),\text{ if }x\in (\alpha_k,\beta_k-t)\text{ or }\beta_k=\infty\mbox{ for some $k\in I$},$$
$$U(t)f(x)=	\sum_{i,\alpha_i>-\infty}b_{i,k}f(\alpha_i+t-(\beta_k-x)),\text{ if }x\in (\beta_k-t,\beta_k), \beta_k<\infty.
$$
We check that $U(t)$ is isometric.
$$\|U(t)f\|^2=\sum_k\int_{(\alpha_k,\beta_k-t)}|f(x+t)|^2\,dx$$$$+\sum_{k,\beta_k<\infty}\int_{(\beta_k-t,\beta_k)}\left|\sum_{i,\alpha_i>-\infty}b_{i,k}f(\alpha_i+t-(\beta_k-x))\right|^2\,dx$$
Use the substitutions $u=x+t$ in the first integral and $u=t-(\beta_k-x)$ in the second; we obtain further:
$$=\sum_k\int_{(\alpha_k+t,\beta_k)}|f(u)|^2\,du+\sum_{k,\beta_k<\infty}\int_{(0,t)}\left|\sum_{i,\alpha_i>-\infty}b_{i,k}f(\alpha_i+u)\right|^2\,du$$
Now use \eqref{eql2}:
$$=\sum_k\int_{(\alpha_k+t,\beta_k)}|f(u)|^2\,du+\int_{(0,t)}\sum_{i,\alpha_i>-\infty}|f(\alpha_i+u)|^2\,du$$$$=\sum_k\int_{(\alpha_k+t,\beta_k)}|f(x)|^2\,dx+\sum_{k,\alpha_k>-\infty}\int_{(\alpha_k,\alpha_k+t)}|f(x)|^2\,dx=\|f\|^2.$$

Next we prove that, for $0\leq t_1,t_2\leq \epsilon$, $U(t_1)U(t_2)=U(t_1+t_2)$; in particular, these operators commute. Let $f\in L^2(\Omega)$ and $x\in \Omega$. If $x\in(\alpha_k,\beta_k-(t_1+t_2))$, then 
$$U(t_1)U(t_2)f(x)=U(t_2)f(x+t_1)=f(x+t_1+t_2)=U(t_1+t_2)f(x).$$
If $x\in (\beta_k-(t_1+t_2),\beta_k-t_1)$, then 
$$U(t_1)U(t_2)f(x)=U(t_2)f(x+t_1)$$$$=\sum_{i,\alpha_i>-\infty}b_{i,k}f(\alpha_i+t_2-(\beta_k-(x+t_1)))=U(t_1+t_2)f(x).$$
If $x\in(\beta_k-t_1,\beta_k)$, then 
$$U(t_1)U(t_2)f(x)=\sum_{i,\alpha_i>-\infty}b_{i,k}U(t_2)f(\alpha_i+t_1-(\beta_k-x))$$$$=\sum_{i,\alpha_i>-\infty}f(\alpha_i+t_1-(\beta_k-x)+t_2)=U(t_1+t_2)f(x).$$

Define now, for $0<t\leq 2\epsilon$, $f\in L^2(\Omega)$, $x\in\Omega$:

$$U(-t)f(x)=f(x-t),\text{ if }x\in (\alpha_i+t,\beta_i)\text{ or }\alpha_i=-\infty\text{ for some }i\in I,$$
$$U(-t)f(x)=	\sum_{k,\beta_k<\infty}\cj b_{i,k}f(\beta_k-(t-(x-\alpha_i))),\text{ if }x\in (\alpha_i,\alpha_i+t), \alpha_i>-\infty.
$$
We check that $U(t)U(-t)=I$. If $x\in(\alpha_k,\beta_k-t)$ then 
$$U(t)U(-t)f(x)=U(-t)f(x+t)=f(x+t-t)=f(x).$$
If $x\in (\beta_k-t,\beta_k)$, then 
$$U(t)U(-t)f(x)=\sum_{i,\alpha_i>-\infty}b_{i,k}U(-t)f(\alpha_i+t-(\beta_k-x))$$$$
=\sum_{i,\alpha_i>-\infty}b_{i,k}\sum_{l,\beta_l<\infty}\cj b_{i,l}f(\beta_l-(t-(\alpha_i+t-(\beta_k-x)-\alpha_i)))
$$$$=\sum_{l,\beta_l<\infty}\sum_{i,\alpha_i>-\infty}b_{i,k}\cj b_{i,l}f(\beta_l-\beta_k+x)
=\sum_{l,\beta_l<\infty}\delta_{k,l}f(\beta_l-\beta_k+x)=f(x).$$
Thus $U(t)$ is a surjective isometry so it is unitary, and $U(-t)=U(t)^*$.

Now let $0\leq t<\infty$. We can write $t=n\epsilon  +r$ with $n\in\bn\cup\{0\}$ and $0\leq r<\epsilon$. Define 
$$U(t)=U(\epsilon)^nU(r).$$
Clearly $U(t)$ is unitary. We check the group property: let $0\leq t_1,t_2$. Write $t_1=n_1\epsilon+r_1$, $t_2=n_2\epsilon+r_2$, $n_1,n_2\in\bn\cup\{0\}$, $0\leq r_1,r_2<\epsilon$. 

$$U(t_1)U(t_2)=U(\epsilon)^{n_1}U(r_1)U(\epsilon)^{n_2}U(r_2)=U(\epsilon)^{n_1+n_2}U(r_1+r_2).$$
If $r_1+r_2<\epsilon$, then $t_1+t_2=(n_1+n_2)\epsilon+(r_1+r_2)$, and using the group property for numbers $\leq\epsilon$ we get $U(t_1)U(t_2)=U(t_1+t_2)$. If $r_1+r_2\geq \epsilon$ , then $t_1+t_2=(n_1+n_2+1)\epsilon+(r_1+r_2-\epsilon)$ and
$$U(t_1+t_2)=U(\epsilon)^{n_1+n_2}U(\epsilon)U(r_1+r_2-\epsilon)=U(t_1)U(t_2).$$
Define,
$$U(-t)=U(t)^* \text{ for }t>0.$$
It is easy to check that $\{U(t)\}$ satisfies the group property. It is also strongly continuous because $U(t)f\rightarrow f$ in $L^2(\Omega)$ when $t\to 0$, and then one can use the group property and the fact that $U(t)$ are unitary. 

The integrability property and the relations \eqref{eq4.2.0} and \eqref{eq4.2.0b} follow directly from the definition of $\{U(t)\}$.

\end{proof}

\newcommand{\fac}{\leftindex^{\textup{f}\alpha} C}
\newcommand{\fbc}{\leftindex^{\textup{f}\beta} C}
\newcommand{\Expa}[1]{\operatorname*{Exp}(#1\vec\alpha)}
\newcommand{\Expb}[1]{\operatorname*{Exp}(#1\vec\beta)}

What is the relation between the measure $\mu$, multiplicity $m$, matrix field $C(\lambda)$ in Theorem \ref{thi2} and the matrix $B$ in Theorem \ref{th4.1} ? We need some notations.

\begin{definition} 
	\label{defca}
	We denote by $\fac(\lambda)$ the matrix obtained from $C(\lambda)$ by removing the column corresponding to the component that has $\alpha_i=-\infty$ (if any). Similarly, we denote by $\fbc(\lambda)$ the matrix obtained from $C(\lambda)$ by removing the column corresponding to the component that has $\beta_k=\infty$. 
	
	For $\lambda\in\brr$, we denote by $\Expa\lambda$ the diagonal matrix with entries $(e^{2\pi i\lambda \alpha_i})_{i\in I, \alpha_i>-\infty}$. We denote by $\Expb\lambda$ the diagonal matrix with entries $(e^{2\pi i\lambda \beta_k})_{k\in I, \beta_k<\infty}$.
\end{definition}

\begin{theorem}\label{thca}
Assume that $l=\inf_i(\beta_i-\alpha_i)>0$. Suppose $H$ is a self-adjoint extension of the partial differential operator $D$ and let $\{U(t)\}$ be the associated unitary group with the integrability property. Let $\mu, m$ and $C$ be as in Theorem \ref{thi2} and let $B$ be as in Theorem \ref{th4.1}. Then 

\begin{equation}\label{eqca1}
	\fbc(\lambda)\Expb{-\lambda}B^T\Expa{\lambda}=\fac(\lambda),\text{ for $\mu$-a.e. $\lambda$}.
\end{equation}
\end{theorem}

\begin{proof}
	We will use the fact that the operator $\mathcal F_C$ transforms $U(t)$ into multiplication by $e^{2\pi i\lambda\cdot t}$ in $L^2(\mu,m)$.
	
	Take $0<\epsilon<l$, $i\in I$ with $\alpha_i>-\infty$, and let $f$ be a function in $C_0^\infty(\Omega)$ supported on $[\alpha_i,\alpha_i+\epsilon]$. Then, by \eqref{eq4.2.0}, 
	the restriction of $U(\epsilon)f$ to the $k$ component $(\alpha_k,\beta_k)$ is $b_{i,k}T(\alpha_i+\epsilon-\beta_k)f$, if $\beta_k<\infty$, and zero when $\beta_k=\infty$. When $\beta_k<\infty$, its Fourier transform is 
	$$\widehat{(U(\epsilon)f)_k}(\lambda)=b_{i,k}e^{2\pi i \lambda\cdot(\alpha_i+\epsilon-\beta_k)}\hat f(\lambda).$$
	Let $\mathfrak a(x)=1$ if $x$ is finite, and $\mathfrak a(x)=0$ if $x$ is infinite, i.e., $x=\pm\infty$. We have: 
	$$\mathcal F_CU(\epsilon)f(\lambda)=C(\lambda)\left(\mathfrak a(\beta_k)b_{i,k}e^{2\pi i\lambda\cdot(\alpha_i+\epsilon-\beta_k)}\hat f(\lambda)\right)_{k\in I }.
	$$
This is the $i$-th column of the matrix $$\hat f(\lambda)e^{2\pi i\lambda\cdot\epsilon}\fbc(\lambda)\Expb{-\lambda}B^T\Expa\lambda.$$ 
	On the other hand, 
	$$e^{2\pi i\lambda\cdot \epsilon}\mathcal F_Cf(\lambda)=e^{2\pi i\lambda\cdot\epsilon}C(\lambda)\begin{pmatrix}
		0\\ \vdots \\ \underset{\text{$i$-th position}}{
			\hat f(\lambda)}\\ \vdots\\0
	\end{pmatrix}$$ 
		This is the $i$-th column of $\hat f(\lambda)e^{2\pi i\lambda\cdot\epsilon}\fac(\lambda)$.
		
		But $\mathcal F_CU(\epsilon)f(\lambda)=e^{2\pi i\lambda\cdot \epsilon}\mathcal F_Cf(\lambda)$ for $\mu$-a.e. $\lambda$ and any $f$ as above, therefore we get  \eqref{eqca1}.
	
\end{proof}

\begin{corollary}
	\label{cor4.4}
	Suppose $\Omega$ is a finite union of intervals, exactly one of which is unbounded. Then there are no self-adjoint extensions of the partial differential operator $D$ on $\Omega$, and there are no unitary groups with the integrability property on $\Omega$. 
\end{corollary}

\begin{proof}
	Assume the unbounded interval of $\Omega$ is $(-\infty,a)$ and the other ones are finite, and the total number of intervals is $N$. 
	
	If, by contradiction, we do have a self-adjoint extension then we do have a unitary group $\{U(t)\}$ with the integrability property. Then, by Theorem \ref{th4.1}, we have a unitary matrix with rows indexed by the finite left-endpoints and columns indexed by the finite right-endpoints. There are $N-1$ finite left-endpoints and $N$ finite right-endpoints. But a unitary matrix must be square, a contradiction. 
\end{proof}

\begin{remark}\label{rem4.4}
	If we allow $\Omega$ to have infinitely many intervals then $\Omega$ can have an unbounded interval and be even spectral (so it does have self-adjoint extensions of the differential operator $D$). Indeed take $\Omega=(-\infty,1)\cup\cup_{n\in\bn}(n,n+1)$. This set is equal to $\brr$ up to measure zero and therefore it is spectral. 
	\end{remark}

	\begin{theorem}\label{th4.5}
	Let $\Omega$ be an unbounded open set in $\mathbb{R}$ with finitely many components. If $\Omega$ is spectral, then $\Omega$ has no gaps between the intervals, so $\Omega$ is equal to $\mathbb{R}$ minus the endpoints of the intervals.
\end{theorem}
\begin{proof}
	Assume $\Omega$ is spectral. By Corollary \ref{cor4.4}, since $\Omega$ is unbounded and has finitely many components, it must contain exactly two unbounded components, which we denote $(a,\infty)$ and $(-\infty,c)$. 
	
	We can assume by contradiction that all gaps between the components of $\Omega$ are positive, i.e., no common end points for the components of $\Omega$, otherwise we join the components with a common endpoint into a single component and the new set $\Omega'$ will be spectral as well, since it differs from $\Omega$ by just a measure zero set. 
	
	Let $\{U(t)\}$ be the group of local translations (by Theorem \ref{thi5}).	Let $\beta$ be the largest right-endpoint less than $a$. Pick a small $\epsilon>0$. Let $\varphi\in C_{0}^{\infty}(\Omega)$ be supported on $(a,a+\epsilon)$. If $\epsilon$ is smaller than the smallest length of the components of $\Omega$, by Theorem \ref{th4.1}, 		
	$U(\epsilon)\varphi$ has support outside $(a,\infty)$ and therefore it is supported in $(-\infty,\beta].$ 
	
	Take $t$ large such that $a-t<c$, then, $(a,a+\epsilon)-(t+\epsilon)\subseteq(-\infty,c)\subseteq\Omega$ and by Lemma \ref{lem2.8}, $U(t+\epsilon)\varphi$ is supported on $[a-t-\epsilon,a-t].$

	On the other hand, $U(\epsilon)\varphi$ is supported on $(-\infty,\beta].$ Take $t$ large so that $\beta-t<c$, i.e., $t>\beta-c.$ Then, with Lemma \ref{lem2.8}, $U(t+\epsilon)\varphi=U(t)(U(\epsilon)\varphi)$ is supported on $(-\infty, \beta-t].$	But for $\epsilon<a-\beta,$ $\beta-t<a-t-\epsilon$ so the two intervals of support are disjoint and  we get a contradiction.  The contradiction shows that we cannot have gaps between intervals, and therefore $\Omega$ is $\brr$ minus the endpoints of the intervals. 
	
\end{proof}

\section{Discrete subgroups of $\brr^d$}\label{sec5}
In this section we will focus on the case when the set $\Omega$ tiles $\brr^d$ with a discrete subgroup $\mathcal T$ of $\brr^d$, and we will show in Theorem \ref{th5.9} that this is equivalent to $\Omega$ having a certain pair measure supported on the dual set $\mathcal T^*$ (Definition \ref{defdu}). In Proposition \ref{prut}, we give a formula for the associated group of local translations in this situation. Our results generalize Fuglede's original result for {\it full-rank} lattices:
\begin{theorem}\cite{Fug74}	\label{thfula}
	Let $\Omega$ be a finite measure set in $\brr^d$, and $A$ a $d\times d$ invertible real matrix.  
	$\Omega$ tiles with the lattice $\mathcal T=A\bz^d$ if and only if $\Omega$ is spectral with spectrum the dual lattice $\mathcal T^*=(A^T)^{-1}\bz^d$.
\end{theorem}

We begin by reviewing some necessary preliminaries. It is well known that any discrete subgroup $\mathcal T$ of $\brr^d$ is isomorphic to $\bz^{d_1}$ for some $d_1\in\bn\cup\{0\}$, which is the dimension of the linear span of $\mathcal T$. Explicitly, there exists linearly independent generators $a^{(1)},\dots,a^{(d_1)}\in\brr^d$ such that the mapping 
$$\bz^{d_1}\ni (n_1,\dots, n_{d_1})\mapsto \sum_{j=1}^{d_1} n_j a^{(j)}\in\brr^d$$
	is a group isomorphism of $\bz^{d_1}$ onto $\mathcal T$. 
	
	One can complete the linearly independent set $\{a^{(1)},\dots, a^{(d_1)}\}$ to a basis $\{a^{(1)},\dots,a^{(d)}\}$ for $\brr^d$. Put the vectors $a^{(j)}$ as columns in a nonsingular real matrix $A$, and then $\mathcal T=A(\bz^{d_1}\times\{0\})$ (here $0$ means that all the coordinates from $d_1+1$ to $d$ are zero). 

The next lemmas are well known and easy to prove:
\begin{lemma}\label{lemti}
	$\Omega$ tiles $\brr^d$ with a set $\mathcal T$ if and only if
	$$\sum_{t\in\mathcal T}\chi_{\Omega}(x-t)=1\mbox{ $\mathfrak m$-a.e. }$$
\end{lemma}
\begin{lemma}
	\label{lemat1}
	Let $\Omega$ be a measurable subset of $\brr^d$, $\mathcal T$ a subset of $\brr^d$, and $A$ be an invertible real $d\times d$ matrix.
	$\Omega$ tiles with $\mathcal T$ if and only if $A^{-1}\Omega$ tiles with $A^{-1}\mathcal T$. 
\end{lemma}
The next lemma is also well known:
\begin{lemma}
	\label{lemti2}
	If $\Omega$ tiles $\brr^d$ with $\bz^d$ (or has spectrum $\bz^d$) then $ \mathfrak m_d(\Omega)=1$.
\end{lemma}
\begin{proof}
	Let $Q$ be the unit cube in $\brr^d$. For each $x\in Q$, up to measure zero, there exists unique $y(x)\in\Omega$ and $k(x)\in \bz^d$ such that $x=y(x)+k(x)$. Define 
	$$Q_k:=\{x\in Q: k(x)=k\},\quad(k\in\bz^d).$$
	Then the sets $\{Q_k\}$ form a partition of $Q$ and the sets $\{Q_k-k\}$ form a partition of $\Omega$ (this can be seen from the fact that the map $Q\ni x\mapsto y(x)\in\Omega$ is a bijection, up to measure zero, which in turn is a consequence of the fact that both $\Omega$ and $Q$ tile by $\bz^d$). Since $Q_k$ and $Q_k-k$ have the same measure, it follows that $\Omega$ has the same measure as $Q$ which is 1. 
\end{proof}

\begin{definition}
	Let $A$ be an invertible real $d\times d$ matrix and $\mu$ a Borel measure on $\brr^d$. We define the measure $A^T\mu$ by $A^T\mu(E)=|\det(A^T)^{-1}|\mu(A^TE)$ for any Borel subset $E$ in $\brr^d$. Equivalently, for any compactly supported continuous function $f$ on $\brr^d$:
	\begin{equation}
		\label{eqdil1}
		\int f\, dA^T\mu=|\det(A^T)^{-1}|\int f((A^T)^{-1}x)\,d\mu(x).
	\end{equation}
	
	Note that if the measure $\mu$ is supported on a set $\Lambda$, then $A^T\mu$ is supported on $(A^T)^{-1}\Lambda$. 
\end{definition}
\begin{lemma}
	\label{lemas}
	If $(\Omega,\mu)$ is a spectral pair, then $(A\Omega,A^T\mu)$ is a spectral pair.
\end{lemma}
\begin{proof}
	Define the operator $\mathcal A: L^2(\Omega)\to L^2(A\Omega)$ by 
	$$\mathcal A f(x)=\sqrt{|\det A^{-1}|}f(A^{-1}x),\quad(x\in A\Omega, f\in L^2(\Omega)).$$
	The operator $\mathcal A$ is unitary. 
	
	Define also the operator $\mathcal A': L^2(\mu)\to L^2(A^T\mu)$ by
	$$\mathcal A'g(x)=\sqrt{|\det A^T|}g(A^Tx),\quad(x\in \brr^d, g\in L^2(\mu)).$$
	We check that $\mathcal A'$ is unitary:
	$$\|\mathcal A'g\|^2=|\det A^T|\int|g(A^Tx)|^2\,dA^T\mu$$$$=|\det A^T||\det (A^T)^{-1}|\int |g(A^T(A^T)^{-1}x)|^2\,d\mu=\int|g|^2\,d\mu.$$
	
	The Fourier transform of $\mathcal Af$, for $f\in L^2(\Omega)$ is 
	$$\mathscr F\mathcal Af(t)=\sqrt{|\det A^T|}\mathscr Ff(A^Tt),\quad (t\in\brr^d).$$
	This means that 
	$$\mathscr F\mathcal A=\mathcal A'\mathscr F.$$
	Since $\mathscr F$ from $L^2(\Omega)$ to $L^2(\mu)$ is unitary, and the operators $\mathcal A$ and $\mathcal A'$ are unitary, it follows that $\mathscr F$ from $L^2(A\Omega)$ to $L^2(A^T\mu)$ is also unitary, and therefore $(A^T\Omega,A^T\mu)$ is a spectral pair. 
\end{proof}

\begin{definition}
	\label{defdu}
	For each discrete subgroup $\mathcal T$ of $\brr^d$ one can define the dual set
	$$\mathcal T^*:=\{\gamma^*\in\brr^d: \gamma^*\cdot \gamma\in\bz \text{ for all }\gamma\in\mathcal T\}.$$
	If $\mathcal T=A(\bz^{d_1}\times\{0\})$ for some invertible real $d\times d$ matrix $A$, and $d=d_1+d_2$, then $\mathcal T^*=(A^T)^{-1}(\bz^{d_1}\times\brr^{d_2})$. Note that the measure $A^T(\delta_{\bz^{d_1}}\times \mathfrak m_{d_2})$ is supported on $\mathcal T^*$. Recall that $\delta_{\bz^{d_1}}$ is the counting measure on the subset $\bz^{d_1}$ of $\brr^{d_1}$.
\end{definition}

\begin{theorem}\label{th5.9}
 The set $\Omega$ tiles with the discrete subgroup $\mathcal T=A(\mathbb{Z}^{d_{1}}\times\{0\})$ if and only if $\Omega$ has pair measure $A^T(\delta_{\mathbb{Z}^{d_{1}}}\times \mathfrak m_{d_{2}})$ which is supported on the dual set $\mathcal T^*$.
\end{theorem}
\begin{proof}
	$\Omega$ tiles with $A(\mathbb{Z}^{d_{1}}\times\{0\})$ iff $A^{-1}\Omega$ tiles with $\mathbb{Z}^{d_{1}}\times\{0\}$.
	$A^{-1}\Omega$ has a pair measure $\delta_{\mathbb{Z}^{d_{1}}}\times \mathfrak m_{d_{2}}$ iff $\Omega$ has pair measure $A^T(\delta_{\mathbb{Z}^{d_{1}}}\times \mathfrak m_{d_{2}}).$ So, considering $A^{-1}\Omega$ instead, it is enough to show that a set $\Omega$ tiles with $\mathbb{Z}^{d_{1}}\times \{0\}$ iff $\Omega$ has pair measure $\delta_{\mathbb{Z}^{d_{1}}}\times \mathfrak m_{d_{2}}$.

	For $x\in \brr^d$ and $d=d_1+d_2$, we write $x\in\brr^d$ as $x=(x_1,x_2)$ with $x_1\in\brr^{d_1}$ and $x_2\in\brr^{d_2}$.

	Assume $\Omega$ tiles with $\mathbb{Z}^{d_{1}}\times\{0\}$. Let $$\Omega_{x_{2}}=\{x_{1}\in \mathbb{R}^{d_1}:(x_{1},x_{2})\in \Omega\},\quad(x_2\in\brr^{d_2}).$$
	Then for almost every $x_{2}\in\mathbb{R}^{d_{2}},$ $\Omega_{x_{2}}$ tiles $\brr^{d_{1}}$ with $\mathbb{Z}^{d_{1}}.$ Indeed, with Lemma \ref{lemti}, we have 
	$$\sum_{k\in\bz^{d_1}}\chi_\Omega(x_1-k,x_2)=1, \mbox{ $\mathfrak m_{d}$-a.e.}$$
	This implies that, for $\mathfrak m_{d_2}$-a.e. $x_2$, $$\sum_{k\in\bz_{d_1}}\chi_{\Omega_{x_2}}(x_1-k)=1,\mbox{ $\mathfrak m_{d_1}$-a.e.,}$$
	and this means that for $\mathfrak m_{d_2}$-a.e. $x_2$, the set $\Omega_{x_2}$ tiles $\brr^{d_1}$ by $\bz^{d_1}$.
	Then, by Theorem \ref{thfula}, it follows that $\Omega_{x_{2}}$ has spectrum $\mathbb{Z}^{d_{1}}$ for $\mathfrak m_{d_2}$-a.e. $x_{2}.$
	With Lemma \ref{lemti2} we also obtain that $\mathfrak m_{d_2}(\Omega_{x_2})=1$.

	Let $\varphi\in L^{1}(\Omega)\cap L^{2}(\Omega).$ Define, $$\varphi_{x_{2}}(x_{1})=\varphi(x_{1},x_{2}).$$
	We have for $(\lambda_{1},\lambda_{2})\in\mathbb{R}^{d_{1}}\times\mathbb{R}^{d_{2}},$
	$$\hat{\varphi}(\lambda_{1},\lambda_{2})=\int_{\mathbb{R}^{d_2}}\int_{\Omega_{x_{2}}}\varphi(x_{1},x_{2})e^{-2\pi i\lambda_{1}\cdot x_{1}}e^{-2\pi i\lambda_{2}\cdot x_{2}}dx_{1}dx_{2}$$$$=\int_{\mathbb{R}^{d_2}}\hat{\varphi}_{x_{2}}(\lambda_{1})e^{-2\pi i\lambda_{2}\cdot x_{2}}dx_{2}.$$
	Define the function $$\phi_{\lambda_{1}}(x_{2})=\hat{\varphi}_{x_{2}}(\lambda_{1}),\quad(\lambda_1\in\brr^{d_1}).$$
	Then, the previous equality means that $$\hat{\varphi}(\lambda_{1},\lambda_{2})=\hat{\phi}_{\lambda_{1}}(\lambda_{2}),\quad (\lambda_1\in\brr^{d_1},\lambda_2\in\brr^{d_2}).$$
	We have 
	$$\int\int|\varphi(x_{1},x_{2})|^{2}dx_{1}dx_{2}=\int_{\mathbb{R}^{d_{2}}}\int_{\Omega_{x_{2}}}|\varphi_{x_{2}}(x_{1})|^{2}dx_{1}dx_{2}$$$$=\int_{\mathbb{R}^{d_{2}}}\sum_{n_{1}\in\mathbb{Z}^{d_{1}}}|\hat{\varphi}_{x_{2}}(n_{1})|^{2}dx_{2},\text{ (as $\Omega_{x_{2}}$ has spectrum $\mathbb{Z}^{d_{1}}$) }$$
	$$=\sum_{n_{1}\in\mathbb{Z}^{d_{1}}}\int_{\mathbb{R}^{d_{2}}}|\hat{\varphi}_{x_{2}}(n_{1})|^{2}dx_{2}=\sum_{n_{1}\in\mathbb{Z}^{d_{1}}}\int_{\mathbb{R}^{d_{2}}}|\phi_{n_{1}}(x_{2})|^{2}dx_{2}$$$$=\sum_{n_{1}\in\mathbb{Z}^{d_{1}}}\int_{\mathbb{R}^{d_{2}}}|\hat{\phi}_{n_{1}}(\lambda_{2})|^{2}d\lambda_{2} \mbox{ (by Plancherel's identity)}$$
	$$=\sum_{n_{1}\in\mathbb{Z}^{d_{1}}}\int_{\mathbb{R}^{d_{2}}}|\hat{\varphi}(n_{1},\lambda_{2})|^{2}d\lambda_{2}=\int_{\mathbb{R}^{d_{1}}\times \mathbb{R}^{d_{2}}}|\hat{\varphi}(x_{1},\lambda_{2})|^{2}d(\delta_{\mathbb{Z}^{d_{1}}}\times \mathfrak m_{d_{2}}(\lambda_2)).$$
	So the Fourier transform is isometric from $L^{2}(\Omega)$ to $L^{2}(\delta_{\mathbb{Z}^{d_{1}}}\times \mathfrak m_{d_{2}}).$
	
	Consider the functions $$e_{n}\otimes f(x_{1},x_{2})=e^{2\pi in\cdot x_{1}}f(x_{2}),$$ 
	where $(x_{1},x_{2})\in \Omega$, $n\in \mathbb{Z}^{d_{1}},$ $f\in L^{1}(\mathbb{R}^{d_{2}})\cap L^{2}(\mathbb{R}^{d_{2}}).$
	
	Then, for $n_1\in\bz^{d_1}$ and $\lambda_2\in\brr^{d_2}$, 
	$$\widehat{e_{n}\otimes f}(n_{1},\lambda_{2})=\int_{\mathbb{R}}\int_{\Omega_{x_{2}}} e^{2\pi in\cdot x_{1}}e^{-2\pi in_1\cdot x_1}dx_{1}\cdot f(x_{2})e^{-2\pi i\lambda_{2}\cdot x_{2}}dx_{2}$$ $$=\int_{\mathbb{R}}\delta_{n,n_{1}}f(x_{2})e^{-2\pi i \lambda_{2}\cdot x_{2}}dx_{2}\text{ ($\Omega_{x_{2}}$ has spectrum $\mathbb{Z}^{d_{1}}$ and with Lemma \ref{lemti2})}$$$$=\delta_{n,n_{1}}\hat{f}(\lambda_{2}).$$
	Thus $\delta_{n,n_{1}}\hat{f}$ is in the range of the Fourier transform $\mathscr F$. Since $\{\hat{f}:f\in L^{1}(\mathbb{R}^{d_{2}})\cap L^{2}(\mathbb{R}^{d_{2}})\}$ is dense in $L^2(\brr^{d_2})$, the range of the transformation is dense in $L^2(\delta_{\bz^{d_1}}\times \mathfrak m_{d_2})$, and therefore it is onto; this proves that $\delta_{\bz^{d_1}}\times \mathfrak m_{d_2}$ is a pair measure for $\Omega$.
	
	For the converse, assume that the set $\Omega$ has pair measure $\delta_{\mathbb{Z}^{d_{1}}}\times \mathfrak m_{d_{2}}.$ We claim that for almost every $x_{2}\in \mathbb{R}^{d_{2}},$ $\Omega_{x_{2}}$ has spectrum $\mathbb{Z}^{d_{1}}.$
	
	Let $\varphi\in L^2(\Omega)\cap L^1(\Omega)$. We have 
	$$\int_{\mathbb{R}^{d_{2}}}\int_{\Omega_{x_{2}}}|\varphi(x_{1},x_{2})|^{2}dx_{1}d{x_{2}}$$$$
	=\sum_{n_{1}\in\mathbb{Z}^{d_{1}}}\int_{\mathbb{R}^{d_{2}}}|\hat{\varphi}(n_{1},\lambda_{2})|^{2}d\lambda_{2}\text{ (since $\delta_{\bz^{d_1}}\times \mathfrak m_{d_2}$ is a pair measure)}
	$$$$=\sum_{n_{1}\in\mathbb{Z}^{d_{1}}}\int_{\mathbb{R}^{d_{2}}}|\hat{\phi}_{n_{1}}(\lambda_{2})|^{2}d\lambda_{2}
	=\sum_{n_{1}\in\mathbb{Z}^{d_{1}}}\int_{\mathbb{R}^{d_{2}}}|\phi_{n_{1}}(\lambda_{2})|^{2}d\lambda_{2}\text{ (Plancherel)}$$
	$$
	=\sum_{n_{1}\in\mathbb{Z}^{d_{1}}}\int_{\mathbb{R}^{d_{2}}}|\hat{\varphi}_{x_{2}}(n_{1})|^{2}dx_{2}	=\int_{\mathbb{R}^{d_{2}}}\sum_{n_{1}\in\mathbb{Z}^{d_{1}}}|\hat{\varphi}_{x_2}(n_{1})|dx_2.
	$$
	Thus 
	\begin{equation}
		\label{*}
		\int_{\mathbb{R}^{d_{2}}}\int_{\Omega_{x_{2}}}|\varphi(x_{1},x_{2})|^{2}dx_{1}d_{x_{2}}
		=\int_{\mathbb{R}^{d_{2}}}\sum_{n_{1}\in\mathbb{Z}^{d_{1}}}|\hat{\varphi}_{x_2}(n_{1})|dx_2.
	\end{equation}

	For $\epsilon>0$ and $a\in\brr^{d_2}$, let $B_\epsilon(a)$ be the ball in $\brr^{d_2}$ centered at $a$ and of radius $\epsilon$. Consider the function:
	$$\tilde{\varphi}_{a,\epsilon}(x_{1},x_2)=\frac{1}{\mathfrak m_{d_2}(B_\epsilon(a))}\chi_{B_\epsilon(a)}(x_{2})\varphi_{x_{2}}(x_{1}),\quad (x_1\in\brr^{d_1},x_2\in\brr^{d_2}).$$
	
	We will apply the relation \eqref{*} to the function $\tilde \varphi_{a,\epsilon}$. We compute 
	$$\widehat{(\tilde{\varphi}_{a,\epsilon})_{x_2}}(n_{1})=\frac{1}{\mathfrak m_{d_2}(B_\epsilon(a))}\chi_{B_\epsilon(a)}(x_{2})\int_{\Omega_{x_{2}}}\varphi_{x_{2}}(x_{1})e^{-2\pi in_{1}\cdot x_{1}}dx_{1}$$$$=\frac{1}{\mathfrak m_{d_2}(B_\epsilon(a))}\chi_{B_\epsilon(a)}(x_{2})\hat{\varphi}_{x_{2}}(n_{1}).$$
	From the relation \eqref{*} we obtain 
	$$\frac{1}{\mathfrak m_{d_2}(B_\epsilon(a))}\int_{B_\epsilon(a)}\int_{\Omega_{x_{2}}}|\varphi_{x_{2}}(x_{1})|^{2}dx_{1}dx_{2}$$$$=\frac{1}{\mathfrak m_{d_2}(B_\epsilon(a))}\int_{B_\epsilon(a)}\sum_{n_{1}\in \mathbb{Z}^{d_{1}}}|\hat{\varphi}_{x_{2}}(n_{1})|^{2}dx_{2}$$
	Using the Lebesgue Differentiation Theorem, let $\epsilon\to 0$ and obtain that 
	$$\int_{\Omega_{a}}|\varphi_{a}(x_{1})|^{2}dx_{1}=\sum_{n_{1}\in \mathbb{Z}^{d_{1}}}|\hat{\varphi}_{a}(n_{1})|^{2}\text{ for a.e $a\in \mathbb{R}^{d_{2}}$}.$$
	This means that $\Omega_a$ has spectrum $\mathbb{Z}^{d_{1}}$ and so it tiles by $\mathbb{Z}^{d_{1}}$ for $\mathfrak m_{d_2}$-a.e. $a$. Then (using Lemma \ref{lemti} as before) $\Omega$ tiles by $\mathbb{Z}^{d_{1}}\times\{0\}.$
\end{proof}

Theorem \ref{th5.9} shows that if a set $\Omega$ tiles with a discrete subgroup then it is also a spectral set. By Theorem \ref{thi5}, there must be an associated unitary group $\{U(t)\}$ of local translations. The next proposition presents an explicit formula for this unitary group. 

\begin{proposition}\label{prut}
	Suppose $\Omega$ tiles $\brr^d$ with a discrete subgroup $\mathcal T$. For almost every $x\in\brr^d$ there are unique $y(x)\in\Omega$ and $\gamma(x)\in\mathcal T$ such that $x=y(x)+\gamma(x)$. Define the operators $\{U(t)\}$ on $L^2(\Omega)$ by
	$$U(t)f(x)=f(y(x+t))=f(x+t-\gamma(x+t)),\quad (t\in\brr^d, x\in\Omega, f\in L^2(\Omega)). $$ 
	
	Then $\{U(t)\}$ is a strongly continuous unitary group of local translations. Its Fourier transform is a multiplication by $e^{2\pi i\gamma^*\cdot t}$ on the dual set $\mathcal T^*$:
	\begin{equation}
		\label{equla}
		\mathscr FU(t)f(\gamma^*)=e^{2\pi i\gamma^*\cdot t}\mathscr F(\gamma^*),\quad(\gamma^*\in\mathcal T^*).
	\end{equation}
	
\end{proposition}

\begin{proof}
	Let $t\in\brr^d$. We check that $U(t)$ is isometric. For $\gamma\in\mathcal T$, denote
	$$\Omega_\gamma:=\{x\in\Omega: \gamma(x+t)=\gamma\}=\{x\in\Omega: x+t-\gamma\in\Omega\}.$$ 
	Since $\Omega$ tiles with $\mathcal T$, the sets $\{\Omega_\gamma : \gamma\in\mathcal T\}$ form a partition of $\Omega$. Also, the sets $\{\Omega_\gamma+t-\gamma: \gamma\in\mathcal T\}$ form a partition of $\Omega$. Indeed, if $(\Omega_\gamma+t-\gamma)\cap (\Omega_{\gamma'}+t-\gamma')\neq \ty$ for $\gamma\neq\gamma'$ in $\mathcal T$, then 
	$(\Omega-\gamma)\cap (\Omega-\gamma')\neq \ty$ (up to measure zero, of course), and this is a contradiction because $\mathcal T$ is a group and $\Omega$ tiles with $\mathcal T$. 
	
	To prove that these sets cover $\Omega$, let $x\in\mathcal T$. Since $\Omega$ tiles with $\mathcal T$ and $\mathcal T$ is a subgroup, there exists $y\in\Omega$ and $\gamma\in\mathcal T$ such that $x-t=y-\gamma$. Then $y+t-\gamma=x$ so $y\in\Omega_\gamma$ and $x=y+t-\gamma\in\Omega_\gamma+t-\gamma$. 
	
	We check that $U(t)$ is isometric. For $f\in L^2(\Omega)$, 
	$$\|U(t)f\|^2=\int_\Omega|f(x+t-\gamma(x+t))|^2\,dx=\sum_{\gamma\in\mathcal T}\int_{\Omega_\gamma}|f(x+t-\gamma)|^2\,dx$$
	$$=\sum_{\gamma\in\mathcal T}\int_{\Omega_\gamma+t-\gamma}|f(u)|^2\,du=\|f\|^2.$$
	
	Next, we check that $U(t)$ has the group property. Let $t_1,t_2\in\brr^d$, $f\in L^2(\Omega)$ and $x\in\Omega$. Let $\gamma_1:=\gamma(x+t_1)$ so that $x+t_1-\gamma_1\in\Omega$. Then 
	$$U(t_1)U(t_2)f(x)=U(t_2)f(x+t_1-\gamma_1).$$
	Let $\gamma_2:=\gamma((x+t_1-\gamma_1)+t_2)$ so that $x+t_1-\gamma_1+t_2-\gamma_2\in\Omega$. This means that $\gamma_1+\gamma_2=\gamma(x+t_1+t_2)$ and 
	$$U(t_2)f(x+t_1-\gamma_1)=f(x+t_1-\gamma_1+t_2-\gamma_2)=U(t_1+t_2)f(x).$$
	Thus $\{U(t)\}$ has the group property and $U(t)$ is an invertible isometry so it is unitary.

	The fact that $\{U(t)\}$ is strongly continuous can be proved as for the classical group of translations $\{T(t)\}$: prove first that, for continuous compactly supported functions $\varphi$ on $\Omega$, $U(t)\varphi\rightarrow U(t_0)$ as $t\to t_0$; then use an approximation argument. 
	
	Note also that if $x\in\Omega$ and $x+t\in\Omega$ then $U(t)f(x)=f(x+t)$, so $\{U(t)\}$ is a group of local translations on $\Omega$. 
	
	We check the relation \eqref{equla}, that the Fourier transform of $U(t)$ is a multiplication operator on $\mathcal T^*$. Let $t\in\brr^d$, $f\in L^2(\Omega)\cap L^1(\Omega)$ and $\gamma^*\in\mathcal T^*$. Using $\gamma^*\cdot\gamma(x+t)\in\bz$, we have
	$$\mathscr FU(t)f(\gamma^*)=\int_\Omega f(x+t-\gamma(x+t))e^{-2\pi i\gamma^*\cdot x}\,dx$$$$=\int_\Omega f(x+t-\gamma(x+t))e^{-2\pi i\gamma^*\cdot(x+t-\gamma(x+t))}e^{2\pi i\gamma^*\cdot t}\,dx$$
	$$=e^{2\pi i \gamma^*\cdot t}\sum_{\gamma}\int_{\Omega_\gamma}f(x+t-\gamma)e^{-2\pi \gamma^*\cdot (x+t-\gamma)}\,dx$$$$=e^{2\pi i\gamma^*\cdot t}\sum_{\gamma\in\mathcal T}\int_{\Omega_\gamma+t-\gamma}f(u)e^{-2\pi i\gamma^*\cdot u}\,du=e^{2\pi i\gamma^*\cdot t}\mathscr Ff(\gamma^*).$$
\end{proof}

{\bf Conflict of interest.} On behalf of all authors, the corresponding author states that there is no conflict of interest.
\bibliographystyle{alpha}	

{\bf Data Availability Statement.} No new data were created or analyzed in this study. Data sharing is not applicable to this article.

{\bf Funding.} No funds, grants, or other support were received during the preparation of this manuscript.

{\bf Author Contribution.} All authors contributed equally to the study conception, design, data collection, analysis, and manuscript preparation. Both authors read and approved the final manuscript.

\end{document}